\documentclass[11pt,a4paper]{article}

\usepackage[utf8]{inputenc}
\usepackage[T1]{fontenc}
\usepackage{textcomp}
\usepackage{amsmath,amsthm,amsopn,amsfonts,amssymb,dsfont,color}
\usepackage{mathrsfs}
\usepackage{tikz}

\usepackage{xcolor}

\usepackage{hyperref}

\usepackage{a4,a4wide}

\def\ep{{\varepsilon}}
\def\R{\mathbb R}
\def\N{\mathbb N}

\newtheorem{theorem}{\textbf{Theorem}}[section]

\newtheorem{lemma}[theorem]{\textbf{Lemma}}
\newtheorem{proposition}[theorem]{\textbf{Proposition}}

\newtheorem{assumption}[theorem]{\textbf{Assumption}}
\theoremstyle{remark}
\newtheorem{remark}[theorem]{\textbf{Remark}}

\numberwithin{equation}{section}

\title{The Bramson correction in the Fisher-KPP equation: from delay to advance
}

\date{}

\begin{document}

\maketitle

\begin{center}
{\large\bf Matthieu Alfaro\footnote{Univ. Rouen Normandie, CNRS, LMRS UMR 6085, F-76000 Rouen, France.}, Thomas Giletti\footnote{Univ. Clermont Auvergne, LMBP UMR 6620, 63178 Aubi\`{e}re, France.} and Dongyuan Xiao\footnote{The University of Tokyo, 3 Chome-8-1 Komaba, Meguro City, Tokyo 153-8914, Japan.} }
\end{center}

\vspace{15pt}

\tableofcontents


\begin{abstract} 
We consider the solution to the scalar Fisher-KPP equation with front-like initial data, focusing on the location of its level sets at large times, particularly their deviation from points moving at the known spreading speed. We consider an intermediate case for the tail of the initial data, where the decay rate approaches, up to a polynomial term, that of the traveling wave with minimal speed. This approach enables us to capture deviations of the form $-r \ln t$ with $r < \frac{3}{2}$, which corresponds to a logarithmic delay when $0 < r < \frac{3}{2}$ and a logarithmic advance when $r < 0$.  The critical case  $r=\frac 32$ is also studied, revealing an extra $\mathcal O(\ln \ln t)$ term. Our arguments involve the construction of new sub- and super-solutions based on preliminary formal computations on the equation with a moving Dirichlet condition. Finally, convergence to the traveling wave with minimal speed is addressed.
\\


\end{abstract}

\section{Introduction}\label{s:intro}

\noindent
This work is concerned with the large-time behavior of $u=u(t,x)$, the solution to the Fisher-KPP equation
\begin{equation}\label{eq}
 \partial _t u   = \partial _{xx} u+u(1-u), \quad t>0,\;  x\in \R,
\end{equation}
starting from a {\it front-like} initial data $u_0$ such that (see  below for a precise statement)
\begin{equation}\label{roughly}
u_0(x) \; \text{ \lq\lq behaves like'' } \;  x^k e^{-x} \text{ as } x\to +\infty, \quad \text{for some } k\geq  -2. 
\end{equation}
As previously established, the level sets of the solution asymptotically move (as $t \to +\infty$) at a constant spreading speed, often denoted by~$c^*$, which is also the minimal traveling wave speed. However, it is expected that the precise location of the level sets deviates slightly from the moving frame associated with this spreading speed. Through PDE techniques, we capture this (logarithmic) correction and, while the solution approaches a family of shifted minimal traveling waves, we find that it may either lag behind (when $-2\leq k<1$) or slightly outpace (when $k>1$) the moving frame with the exact speed~$c^*$.

\medskip

It is a well-established \cite{Aro-Wei-78}, \cite{Rothe-78} that, for large classes of front-like initial data (e.g. compactly supported or exponentially decaying as $x\to +\infty$), the solution of the Fisher-KPP equation~\eqref{eq} spreads with some speed $c\geq c^*=2$, in the sense that
\begin{equation}\label{spreading-bis}
\min _{x\leq c_1t} u(t,x) \to 1 \text{ as  } t\to +\infty \text{ if }
c_1<c, \quad \max _{x\geq c_2t} u(t,x)\to 0 \text{ as  } t\to +\infty
\text{ if } c_2>c.
\end{equation}
The steepness of the right tail of the initial data significantly influences the position of the level sets. In particular, if the initial data has a ``heavy tail'', meaning it decays more slowly than any exponential as $x \to +\infty$, then the spreading speed is infinite and acceleration occurs~\cite{Ham-Roq-10}.

Furthermore, in the case of exponential decay, and using a similar terminology to that of~\cite{Van-03}, one should distinguish between the {\it flat} case
$$
u_0(x) \; \text{ \lq\lq behaves like'' } \;  e^{-\lambda x} \text{ as } x\to +\infty, \quad \text{for some } \lambda<1, 
$$
and the {\it sufficiently steep} case
$$
u_0(x) \; \text{ \lq\lq behaves like'' } \;  e^{-\lambda x} \text{ as } x\to +\infty, \quad \text{for some } \lambda>1.
$$ 
In the flat case, the spreading speed is given by the dispersion relation $c= \frac{\lambda^2 +1}{\lambda}>c^*=2$, and there is no logarithmic correction. Specifically, the position of any level set of the solution behaves as $ct + \mathcal{O} (1)$ as $t \to +\infty$. We refer to \cite{Rothe-78}, \cite{Uch-78} for even stronger results on the convergence to a traveling wave. On the other hand, in the sufficiently steep case, the spreading speed is $c=c^*=2$, but the position of any level set of the solution behaves as $2t - \frac{3}{2}\ln t + \mathcal{O}(1)$ as $t\to +\infty$, indicating a logarithmic delay. This was first known through probabilistic arguments \cite{Bra-83, Bra-78}, and is commonly referred to as the logarithmic Bramson correction. As a result, the solution converges in the appropriate moving frame to the traveling wave with minimal speed. This convergence can also be proved by other methods, e.g. based on an extended maximum principle related to the nonincrease of the number of intersection points between any two solutions of the Cauchy problem~\cite{Kol-Pet-Pis-37,Lau-85}. Additionally, let us mention \cite{Bru-Der-97}, \cite{Ebe-Saa-00}, \cite{Van-03} for an analysis of such correction terms through convincing formal asymptotic expansions.

More recently, logarithmic corrections have been revisited through PDE techniques by  Hamel, Nolen, Roquejoffre, and Ryzhik. In \cite{Ham-Nol-Roq-Ryz-13}, they consider the case where $u_0\equiv 0$ on a neighbourhood of $+\infty$ (roughly corresponding to $\lambda=+\infty$). However, a careful analysis of their proof reveals that it also applies to the case where
$$
\Big( x \mapsto xe^{x}u_0(x) \Big) \in L^1(0,+\infty),
$$
which, in particular, includes the sufficiently steep case. Further refinements, including the convergence rate, have then been developed by some of the same authors~\cite{Nol-Roq-Ryz-17, Nol-Roq-Ryz-19}. In the same PDE framework, the connection between logarithmic corrections and the nature of waves (pushed vs. pulled) was explored in \cite{Gil-21}, \cite{An-23, An-Hen-Ryz-23}. Additionally, the effect of nonlocal terms or discrete diffusion has recently drawn significant attention~\cite{Bou-Hen-Ryz-20}, \cite{Gra-22}, \cite{Bou-24}, \cite{BFRZ}.

Our goal here is to investigate the intermediate case \eqref{roughly}. The fact that the solution spreads with speed $2$, in the sense that \eqref{spreading-bis} holds with $c=c^*=2$, follows directly from the previously mentioned results and the parabolic comparison principle, see also the insightful introduction in~\cite{Ham-Nad-12}. However, we will further prove, through PDE techniques akin to those in the aforementioned works, that such special initial conditions also lead to convergence to a traveling wave with a logarithmic correction, but with a prefactor depending (in value and sign) on $k> -2$. The critical case $k=-2$ will be tackled separately, as the position of the level sets in this case involves an extra term of order $\ln \ln t$.

\medskip

Through this work, unless otherwise specified, we always denote by $u=u(t,x)$ the solution of \eqref{eq} with initial condition $u_0$ satisfying the following assumption.

\begin{assumption}[Initial condition]\label{ass:initial} The function $u_0 :\R\to \left[0,1\right]$ is uniformly continuous, positive, satisfies  $\liminf _{x\to -\infty} u_0(x)>0$ and there are $k \geq -2$, $0<a\leq A$ such that
\begin{equation}\label{tail-u0}
a x^ke^{-x}\leq u_0(x)\leq  A x^k e^{-x}, \quad \forall x>1.
\end{equation}
\end{assumption}
Notice that, by applying this assumption and the comparison principle, we immediately obtain
$0<u(t,x)<1$ for all $(t,x)\in(0,+\infty)\times\R$.

To state our results, we denote $U$ the traveling wave with minimal speed (or simply minimal traveling wave) of \eqref{eq}, which satisfies the following conditions:
\begin{equation}\label{def:tw}
\left\{\begin{aligned}
&U''+2U'+U(1-U)=0 \, \text{ on } \R,\\
&U(-\infty)=1,\ U(+\infty)=0,\ U'<0. 
\end{aligned}
\right.
\end{equation}
In particular $U (x-2t)$ solves \eqref{eq}. The existence and uniqueness (up to translation) of such a traveling wave were established in~\cite{Fis-37}, \cite{Kol-Pet-Pis-37}, \cite{Aro-Wei-78}. It is also known that 
$$
U (x) \sim B x e^{-x}, \ \text{ as } x\to +\infty,
$$
for some $B >0$. In this context, when $k=1$, respectively $k<1$, respectively $k>1$, Assumption~\ref{ass:initial} means that the right tail of the initial data roughly behaves like, is lighter than, or is heavier than that of the minimal traveling wave, respectively.

\medskip

Our main result for the noncritical case $k>-2$ is as follows. 

\begin{theorem}[Bramson correction when $k> -2$]\label{th:bra}   
Let $u=u(t,x)$ be the solution of \eqref{eq} with initial data~$u_0$ satisfying Assumption~\ref{ass:initial} with $k > -2$. Define $$r:=\frac{1-k}{2}<\frac 3 2 .$$
Then,
there exists a constant $C\geq 0$ such that
\begin{equation}\label{asy}
\lim_{t\to+\infty}\; \inf_{|h|\le C} \; \Big\Vert u(t,\cdot)-U(\cdot-2t+r\ln t+h)\Big\Vert_{L^{\infty}(0,+\infty)}=0.
\end{equation}
\end{theorem}

Most of the proof will be devoted to analyzing the position of the level sets, from which the convergence part largely follows by a Liouville-type result as in~\cite{Ham-Nol-Roq-Ryz-13}. More precisely, we will prove that, for any $0<m<1$,  there are $C_m,T_m>0$ such that
$$
E_m(t)\subset\Big(2t-r\ln t-C_m, 2t-r\ln t+C_m\Big), \text{ for all } t\geq T_m,
$$
where
\begin{equation}\label{defi-level-set}
E_m(t):=\{x\in\R:\, u(t,x)=m\}
\end{equation}
denotes the $m$-level set of $u(t,\cdot)$
which, at least for large enough times, is not empty. As established in the aforementioned literature, the spreading result \eqref{spreading-bis} holds with $c=c^*=2$.

Let us briefly discuss the conclusions of  Theorem~\ref{th:bra}. 
\begin{itemize}
\item For the case $k>1$, the Bramson correction implies that the precise location of the front of the solution moves faster than $x=2t$. Therefore, we call it the {\it advance case}.
\item Conversely, if $-2<k<1$, the Bramson correction implies that the precise location of the front of the solution moves slower than $x=2t$. Hence, we refer to this as the {\it delay case}. 

\item When $k=1$, there is no logarithmic deviation, which is consistent with the observation that, in this case,  the right tail of the initial data is \lq\lq comparable'' to that of the traveling wave.
\end{itemize}

\medskip

As for the critical case $k=-2$, we find that $r = \frac{1-k}{2} = \frac{3}{2}$. In this scenario, however, the drift is not expected to follow a purely logarithmic form. Instead, we provide the following result regarding the large-time asymptotics of the solution, which includes an extra $\mathcal O(\ln \ln t)$ term.

\begin{theorem}[The critical case $k=-2$]\label{th:critical}
Let $u=u(t,x)$ be the solution to \eqref{eq} starting from $u_0$ satisfying Assumption \ref{ass:initial} with $k=-2$.

Then, there exists a constant $C \geq 0$ such that
\begin{equation}\label{asy2}
\lim_{t\to+\infty}\; \inf_{|h|\le C} \; \Big\Vert u(t,\cdot)-U \Big( \cdot-2t +\frac{3}{2}\ln t- \ln \ln t+h \Big) \Big\Vert_{L^{\infty}(0,+\infty)}=0.
\end{equation}
\end{theorem}

We refer to the work \cite{Ber-Bru-17}, based on probabilistic arguments, for refined results in the critical case (but when the problem is instead posed on a moving half-line with a Dirichlet boundary condition). Let us also refer to the comprehensive work \cite{Ebe-Saa-00} which primarily utilizes formal asymptotic expansions. In contrast, our approach is distinct and relies on PDE techniques.

\paragraph{Outline of the paper.} The organization of this work is as follows. In Section \ref{s:formal} we present some formal computations that are crucial to \lq\lq guess'' the shape of the solution, and provide material to construct sharp sub- and super-solutions. Building on these computations, we obtain the upper estimate on the level sets for the case $k> -2$ in Section \ref{s:upper} and the lower estimate in Section~\ref{s:lower}. The position of the level sets in the critical case $k=-2$ is studied in Section \ref{s:critical}, using a similar but slightly different ansatz satisfying a moving Dirichlet boundary condition. Last, the convergence to the profile of the minimal traveling wave is addressed in Section~\ref{s:proof of th}. 

\section{Enlightening formal computations}\label{s:formal}

The computations presented here are the keystone for constructing sharp sub- and super-solutions in the subsequent sections. 

\paragraph{Computations involving the minimal wave.} Denote as $U$ the minimal Fisher-KPP wave, i.e. the (unique up to shifts) solution of~\eqref{def:tw}.
Then, defining
$
v(t,x):=U\left(x-2t+r\ln (t+t_0)\right)
$
with $t_0>0$,
we immediately compute
$$
\mathcal{L}v:= \partial _t v -\partial_{xx}v-v(1-v)=\frac{r}{t+t_0}U'\left(x-2t+r\ln (t+t_0)\right).
$$
When $r>0$ (i.e. $k = 1 - 2r <1$), then $v$ acts as a sub-solution of~\eqref{eq}. However, due to the asymptotics $U (x) \sim x e^{-x}$ as $x \to +\infty$, it follows that for any $t\geq 0$ that $v(t,x) > u_0 (x)$ on some right half-line. Conversely, when $r<0$ (i.e. $k>1$), then $v$ acts as a super-solution, but $v(t,x)$ becomes smaller than $u_0(x)$ for large $x$. Consequently, in both cases, the comparison principle seems inapplicable.

\begin{lemma}\label{lem:tw-sub-super}
Let $U$ be the minimal traveling wave and let $t_0>0$ be given. 

If  $k\geq 1$, then $U(x-2t+r\ln (t+t_0))$ is a super-solution of \eqref{eq} on $[0,+\infty)\times \R$. 

If  $k\leq 1$, then $U(x-2t+r\ln (t+t_0))$ is a sub-solution of \eqref{eq} on $[0,+\infty)\times \R$.
\end{lemma}

This indicates that the traveling wave may serve in the so-called front zone $x\leq 2t+\mathcal \mathcal{O}(\sqrt t)$. However, a different tool is required in the so-called far away zone $x\geq 2t +\mathcal \mathcal{O}(\sqrt t)$. To do so, we work in the moving frame as in the seminal work \cite{Ham-Nol-Roq-Ryz-13}, see also \cite{Gil-21}, leading us to an ODE problem. In contrast with previous literature, the adequate solution is determined by the choice of $k> -2$ in \eqref{tail-u0}, which serves as a good starting point for constructing a relevant ansatz.

\paragraph{Computations in the moving frame.} Following the idea from \cite{Ham-Nol-Roq-Ryz-13}, we start from the linearized equation of \eqref{eq} at $u\approx 0$,
\begin{equation*}
 \partial _t u= \partial _{xx} u+u.
\end{equation*}
By changing the variable
\begin{equation*}
z=x-(2t-r\ln(t+t_0)),
\end{equation*}
we reach
\begin{equation*}
 \partial _t u= \partial _{zz} u+\left(2-\frac{r}{t+t_0}\right)\partial _z u+u.
\end{equation*}
Letting
\begin{equation*}
v(t,z):=e^z u(t,z),
\end{equation*}
we get
\begin{equation*}
 \partial _t v= \partial _{zz} v+\frac{r}{t+t_0}(v-\partial _z v).
\end{equation*}
Using the self-similar variables
\begin{equation*}
\tau=\ln (t+t_0)-\ln t_0,\quad y=\frac{z}{\sqrt{t+t_0}},
\end{equation*}
we reach
\begin{eqnarray*}
\partial_\tau v &=&\partial_{yy} v+\frac y 2 \partial_y v+rv-\frac r{\sqrt{t_0}}e^{-\frac \tau 2} \partial_y v\\
&\approx & \partial_{yy} v+\frac y 2 \partial_y v+rv,
\end{eqnarray*}
by formally ignoring the non-autonomous exponentially decreasing drift term. Plugging the ansatz
$$
v(\tau, y ) = e^{\frac{\tau}{2}} w (y),
$$
with $w (0) = 0$ and $w' (0) = 1$ (which can be motivated by a matching argument with the traveling wave), we are left to solve the Cauchy problem
\begin{equation*}
\left\{
\begin{array}{l}
 w''  + \frac{y}{2} w' + \left( r - \frac{1}{2} \right) w = 0, \vspace{3pt}\\
w(0) = 0, \quad w' (0) = 1.
\end{array}
\right.
\end{equation*}

\begin{lemma}\label{lem:ode-w}
The solution $w$ of the Cauchy problem
\begin{equation}\label{main_ode}
\left\{
\begin{array}{l}
w'' + \frac{y}{2} w' + \left( r - \frac{1}{2} \right) w = 0, \quad y>0, \vspace{3pt}\\
w(0) = 0, \vspace{3pt}\\
w' (0) = 1,
\end{array}
\right.
\end{equation}
satisfies the following properties:
\begin{enumerate}
\item [(i)] if $r = \frac{3}{2}$, then $w (y) = y e^{-y^2/4}$;
\item [(ii)]  if $r < \frac{3}{2}$, then $w >0$ in $(0,+\infty)$, $\lim_{y \to +\infty} \frac{w'}{w} (y) =0 $ and there exists $C >0$ such that 
\begin{equation}\label{w-infinity}
w(y) \sim C y^{k} \mbox{ as } y \to +\infty, \quad k = 1 -2 r.
\end{equation}
\end{enumerate}
\end{lemma}

\begin{proof} The case $r=\frac 32$ is obvious as one may directly check that $y e^{-y^2/4}$ indeed satisfies the wanted problem. When $r<\frac 32$, we may write $w(y)=y\varphi(y)$ where 
$$
\left\{
\begin{array}{l}
\varphi ''+\left(\frac{y}{2}+\frac{2}{y}\right)\varphi'+r\varphi=0, \quad y>0, \vspace{3pt}\\
\varphi(0) = 1, \vspace{3pt}\\
\varphi' (0) = 0.
\end{array}
\right.
$$
Notice that, using a Sturm-Liouville approach, one can recast the above ODE problem into an integral equation and prove the existence and uniqueness of a local solution which satisfies $\varphi''(0)=-\frac r 3$, see \cite[Proposition 3.1]{Har-Wei-82}. Next, by writing $\varphi(y)=e^{-\frac{y^2}{4}}\psi(\frac{y^2}{4})$, we see that $\psi$ satisfies
$$
\left\{
\begin{array}{l}
 z\psi''+\left(\frac 3 2-z\right)\psi'+\left(r-\frac{3}{2}\right)\psi=0, \quad z>0, \vspace{3pt}\\
\psi(0) = 1, \vspace{3pt}\\
\psi' (0) = 1-\frac{2r}{3}.
\end{array}
\right.
$$
Then $\psi(z)={}_1F_1\left(\frac{3-2 r}{2},\frac 3 2,z\right)$, where ${}_1 F_1(a,b,z)$ denotes
the confluent hypergeometric function of the first kind, or Kummer's function~\cite{Abr-Ste-64}. According to \cite[formula 13.1.4]{Abr-Ste-64}, for $a$  not a nonpositive integer, we have
$$
{}_1F_1(a,b,z)\sim \frac{\Gamma(b)}{\Gamma(a)}
\frac{e^z}{z^{b-a}}, \quad\text{ as } z\to +\infty,
$$
which, applied to  $w(y)=ye^{-\frac{y^2}{4}}\psi\left(\frac{y^2}{4}\right)$, exactly corresponds to \eqref{w-infinity}. Additionally, from \cite[formula 13.4.14]{Abr-Ste-64}, we have
$$
\sqrt z\left(\frac{ \frac{d}{dz} {}_1F_1(a,b,z)}{  {}_1F_1(a,b,z)} -1\right)\to 0 , \quad\text{ as } z\to +\infty,
$$
which implies $\frac{w'(y)}{w(y)}\to 0$ as $y\to+\infty$. Lastly, since $a=\frac{3-2r}{2}>0$, $b=\frac 32$, it follows from \cite[formula 13.1.2]{Abr-Ste-64} that $\psi$ is positive on $(0,+\infty)$ and so is $w$.
\end{proof}

Observe that, returning to $u=u(t,z)$, we have
$$
u(t,z)=e^{-z}\sqrt{\frac{t+t_0}{t_0}} w\left(\frac{z}{\sqrt{t+t_0}}\right),
$$
so that not only the asymptotic behavior \eqref{w-infinity} ensures that
$$
u(0,z)\sim \frac{C}{t_0^{k/2}}z^ke^{-z}, \quad \text{ as } z\to +\infty,
$$ 
which matches with the spatial decay of the initial data as stated in Assumption \ref{ass:initial}, but also ensures
$$
u(t,z)\sim \frac{1}{\sqrt{t_0}}ze^{-z}, \quad \text{ as } t\to +\infty,
$$
which matches with the asymptotics of the minimal traveling wave.  This suggests that this ansatz may serve effectively as either a sub-solution or a super-solution.

\begin{remark}\label{rem:r'}
Incorporating the nonlinear and non-autonomous terms that were previously ignored in the formal computation, the proposed ansatz, unfortunately, fails to generate either a valid sub-solution or super-solution. To overcome this issue, we will consider a slightly different approach, introducing a logarithmic drift different from the expected one. More precisely, taking $r' >0$ and letting instead
$z = x - (2t - r' \ln (t+t_0))$ and
$$v (\tau, y) = e^{ (r' - r ) \tau}  e^{\frac{\tau}{2}} w(y)$$
in the above computation, we observe that $w$ still satisfies the ODE~\eqref{main_ode}. Thanks to this additional parameter~$r'$, the resulting alternative ansatz now satisfies a Dirichlet boundary condition at $x = 2t + r' \ln (t+t_0)$, and still exhibits similar asymptotic behavior. However, the inclusion of $r'$ proves to be more convenient to construct sub- and super-solutions of~\eqref{eq}.
\end{remark}

\section{The upper estimate}\label{s:upper}

This section focuses on establishing the upper estimate for the position of the $m$-level set $E_m$ of the solution, as defined in~\eqref{defi-level-set}.

\begin{proposition}[Upper estimate]\label{prop:super-nancy}
Let $u=u(t,x)$ be the solution to \eqref{eq} starting from $u_0$ satisfying Assumption \ref{ass:initial} with $k>-2$. Let $m\in(0,1)$ be given. Then, there are $C_m>0$ and $T_m>0$ such that
\begin{equation}
\label{conclusion-super-nancy}
E_m(t)\subset (-\infty,2t-r\ln t +C_m), \quad \forall t\geq T_m,
\end{equation}
with
$$r =\frac{1-k}{2} < \frac{3}{2}.$$
Moreover, there are $k_1>0$ and $\sigma_1>0$ such that
\begin{equation}\label{upper estimate front}
u(t,2t-r\ln t+y)\le k_1 (y+1) e^{-y}\quad \text{for any  $t\geq 1$, $0 \le y\le \sigma_1\sqrt{t}$}.
\end{equation}
\end{proposition}

\subsection{Proof of Proposition \ref{prop:super-nancy}}\label{ss:proof-upper}

To construct an appropriate ansatz for the solution, we begin by considering the form derived from the formal computations in the moving frame (see Section \ref{s:formal}), which suggests the following behavior:
$$
e^{-z} \sqrt{t} w\left( \frac{z}{\sqrt{t}} \right),\quad  z=x-2t + r \ln t,
$$
with~$w$ coming from Lemma~\ref{lem:ode-w}. However, some modifications will be necessary to deal, in particular, with the non-autonomous drift term that was neglected in Section~\ref{s:formal}, see Remark~\ref{rem:r'}. However, this ansatz ignores the non-autonomous drift term that we neglected earlier in the derivation, and incorporating this term will be crucial to constructing proper sub- and super-solutions. These modifications slightly differ between the delay case and the advance case, as the signs of those ignored terms are different. 

\paragraph{Step 1: a first super-solution.} For $t>0$, $z \geq 0$, we define
 \begin{equation}\label{def:psi}
\psi (t,z) :=  e^{-z} t^{\frac 1 2 + r' - r} \, w \left( \frac{z}{\sqrt{t}}\right),
\end{equation} 
where 
the function $w$ is given by Lemma~\ref{lem:ode-w}, and
\begin{equation}\label{r'r}
r ' = \max (r, 0).
\end{equation}
For $M\geq0$ to be specified later, let us prove that 
\begin{equation}
\label{supersol}
\widehat  u(t,x):=\left(1-\frac{M}{\sqrt t}\right)\psi(t,x-2t+r' \ln t)
\end{equation}
is a super-solution of \eqref{eq} for all $t\geq t_0$ (with $t_0>1$ large enough) and $x \geq 2t - r' \ln t$.

Let us start with
$$
\tilde{u} (t,x) := \psi(t,x-2t+r'\ln t).
$$
Observe that
\begin{eqnarray*}
\partial_t \psi (t,z) & =& \left( \frac{1}{2} + r' - r \right) \, \frac{e^{-z}}{t^{\frac 12 +r - r'}} \, w \left( \frac{z}{\sqrt{t}}\right) - \frac{z e^{-z}}{2 t^{1+r -r'}} \,  w ' \left( \frac{z}{\sqrt{t}}\right),\\
\partial_z \psi (t,z)& = & - \psi (t,z) + e^{-z} t^{r'-r} \, w' \left( \frac{z}{\sqrt{t}}\right), \\
\partial_{zz} \psi (t,z) 
& = &  \psi (t,z) - 2 e^{-z} t^{r' -r } \, w' \left( \frac{z}{\sqrt{t}} \right) + \frac{e^{-z}}{t^{\frac 12 +r - r'}} \, w '' \left( \frac{z}{\sqrt{t}}\right).
\end{eqnarray*}
Combining this and \eqref{main_ode},  we reach
\begin{eqnarray*}
\mathcal L \tilde u (t,x) &\geq & \partial_t \tilde{u} - \partial _{xx} \tilde u -\tilde u \\
& =& \partial _t \psi+\left(-2+\frac{r'}{t}\right)\partial _z \psi -\partial_{zz}\psi -\psi\\
& = &  e^{-z} \left( \frac{ \frac{1}{2} + r' - r}{t^{\frac 12 + r - r'}} - \frac{r'}{t^{\frac 12 +r -r'} } \right) \, w  \left( \frac{z}{\sqrt{t}} \right)  + e^{-z} \left(- \frac{z}{2t^{1+r -r'}}+ \frac{r'}{t^{1+r - r'}} \right) \, w '  \left( \frac{z}{\sqrt{t}} \right) \\
& & \qquad - \frac{e^{-z}}{t^{\frac 12 + r - r' }} \, w ''  \left( \frac{z}{\sqrt{t}} \right) \\
& = &  e^{-z} \frac{ \frac{1}{2} - r}{t^{\frac 12 + r - r'}}   \, w  \left( \frac{z}{\sqrt{t}} \right)  + e^{-z} \left(- \frac{z}{2t^{1+r -r'}}+ \frac{r'}{t^{1+r - r'}} \right) \, w '  \left( \frac{z}{\sqrt{t}} \right) \\
& & \qquad + \frac{e^{-z}}{t^{\frac 12 + r - r' }} \,  \left( \frac{z}{2 \sqrt{t}} \, w' \left( \frac{z}{\sqrt{t}} \right) + \left(r-\frac{1}{2} \right) \, w \left( \frac{z}{\sqrt{t}} \right) \right) \\
& = &  \frac{r' e^{-z}}{t^{1+r - r'}}  \, w '  \left( \frac{z}{\sqrt{t}} \right),
\end{eqnarray*}
where we use the shorthand $z=x-2t+r'\ln t$.  

When $r \leq 0$, then $r' = 0$ and $\tilde{u}$ is already a super-solution, so one may choose $M=0$ in~\eqref{supersol}.  
Unfortunately, when $r >0$ this remaining term may be negative, and at least this is the case when $r > \frac{1}{2}$ due to $w (y) \sim C y^{1-2r}$  as $y \to +\infty$. 

We now deal with the case $r>0$ (so that $r'=r$).
The key point is that $w$ is increasing on a neighborhood of 0, beyond which $\vert w'\vert/w$ remains bounded. Precisely, from Lemma \ref{lem:ode-w}, we can select $\delta >0$ such that $w' >0$ in $[0,\delta)$, and then $M >0$ such that
\begin{equation}\label{w'w}
\frac{2 r' |w'|}{w} \leq M \quad \mbox{ on } \ [\delta,+\infty).
\end{equation}
Now, we set $t_0>1$ large enough such that $1-\frac{M}{\sqrt t_0}\geq \frac 12 $. Then, since $\widehat{u}(t,x)=\left(1-\frac{M}{\sqrt t}\right) \tilde u (t,x)$, one has
$$
\mathcal L \widehat u \geq  \partial_t \widehat {u} - \partial _{xx} \widehat u -\widehat u= \frac{e^{-z}}{t} \, \left[\left(1-\frac{M}{\sqrt t}\right) r'  w '  \left( \frac{z}{\sqrt{t}} \right)  + \frac{M}{2}  \, w \left( \frac{z}{\sqrt{t}}\right) \right]\geq 0,
$$
for any $t\geq t_0$, $z=x-2t+r'\ln t\geq 0$. The last inequality follows from either $w' \left( \frac{z}{\sqrt{t}} \right) >0$ for $0\leq z \leq \delta \sqrt{t}$, or from \eqref{w'w} for $z \geq \delta \sqrt{t}$. We have thus proved that $\widehat {u}=\widehat u(t,x)$ is a super-solution of \eqref{eq} for all $t\geq t_0$ and $x \geq 2t - r' \ln t$.

\paragraph{Step 2: comparison with the solution.} Since $\widehat{u}$ from \eqref{supersol} satisfies $\widehat{u} (t,2t - r' \ln t)=0$ for $t\ge 0$, we cannot directly use it as a super-solution. Therefore we will first cut it by $1$ on a left half-line. By setting $t_0 >1$ large enough such that $\frac{M}{\sqrt{t_0}} \leq \frac{1}{2}$, we have
\begin{eqnarray*}
&&\widehat{u} \left(t+t_0, 2(t+t_0)-r\ln (t+t_0) +  1\right )\\
&& \qquad = \left( 1 - \frac{M}{\sqrt{t+t_0}} \right)   e^{-1- (r' -r) \ln (t+t_0)} \, (t+t_0)^{\frac 12 +r' - r}  w \left( \frac{(r'-r) \ln (t+t_0) +1}{\sqrt{t+t_0}} \right)\\ 
&& \qquad \geq \frac{e^{-1}}{2} \sqrt{t+t_0}\times  w \left( \frac{(r'-r) \ln (t+t_0) +1}{\sqrt{t+t_0}} \right).
\end{eqnarray*}
We deduce from $w'(0)=1$ that
$$
w \left( \frac{(r'-r) \ln (t+t_0) +1}{\sqrt{t+t_0}} \right) \sim  \frac{(r'-r) \ln (t+t_0) +1}{\sqrt{t+t_0}} \quad  \text{ as } t\to +\infty,
$$
hence (using also the fact that $r'-r \geq 0$)
$$\inf_{t \geq 0 } \ \widehat{u} \left(t+t_0,2(t+t_0)-r\ln (t+t_0) +  1\right) > 0.$$
It follows that we can find $K\gg 1$ so that
$$
K \widehat{u} \left(t+t_0 , 2(t+t_0)-r\ln (t+t_0) +  1\right) > 1,
$$
for all $t \geq 0$.
Then
$$\overline{u} (t,x) :=\left\{
\begin{array}{ll}
1 & \ \mbox{ if }\ x < 2(t+t_0) -r \ln (t+t_0) +1 , \vspace{3pt} \\
\min \{1, K \widehat{u} (t+t_0, x ) \} & \ \mbox{ if } \ x \geq 2(t+t_0) -r \ln (t+t_0) +1 ,\end{array}
\right.
$$
is a generalized super-solution of~\eqref{eq}. 

Next we prove that, up to increasing $K>0$ if necessary, 
\begin{equation}\label{comptime0}
u_0(x) \leq \overline{u} (0, x).
\end{equation}
Since $u_0 \leq 1$, we only have to check that
\begin{equation}\label{compt0super}
\forall x \geq  2t_0- r \ln t_0 + 1, \quad u_0 (x) \leq K \widehat{u} (t_0, x).
\end{equation}
By \eqref{def:psi}, \eqref{supersol}, we have  
$$\widehat{u} (t_0,  x) \geq C (t_0) e^{-x}  w\left( \frac{x-2t_0 +r' \ln t_0}{\sqrt{t_0}} \right),$$
for some $C(t_0) >0$. Then, due to 
$w (y) \sim C y^k$ as $y \to +\infty$ with $C>0$ according to Lemma~\ref{lem:ode-w}, up to decreasing the value of the constant $C(t_0)$ we get that
$$
\forall x \geq  2t_0- r \ln t_0 + 1, \quad \widehat{u} (t_0,  x ) \geq C (t_0) x^k e^{-x}.
$$
From this and \eqref{tail-u0} it follows that, for $K>0$ large enough, \eqref{compt0super} holds, and so does \eqref{comptime0}. However, we point out that the assumption $r< \frac{3}{2}$ played a crucial role here, since when $r = \frac{3}{2}$ then $w (y)$ has a gaussian behavior as $y \to +\infty$ and \eqref{compt0super} cannot hold regardless of the multiplicative constant~$K$.

Still, when $r< \frac{3}{2}$, and applying the parabolic comparison principle, we get that
\begin{equation}\label{comparisonsuper1}
u (t,x) \leq \overline{u} (t,x),
\end{equation}
for any $t\geq 0$ and $x \in \mathbb{R}$.

In particular, for $t$ large enough and $x \geq 2 (t+t_0) - r \ln (t+t_0) + \sqrt{t}$, we find that
\begin{eqnarray*}
u(t,x)&  \leq &  K e^{-(x-2 (t+t_0) + r' \ln (t+t_0))} (t+t_0)^{\frac{1}{2} + r' - r} w \left( \frac{x-2 (t+t_0) + r' \ln (t+t_0)}{\sqrt{t+t_0}} \right) \\
&  \leq &   K \left( \sup_{y \geq 1} \frac{w (y)}{y^k} \right) e^{-(x-2 (t+t_0))} (t+t_0)^{\frac{1}{2}  - r}  \times \left( \frac{x-2 (t+t_0) + r' \ln (t+t_0)}{\sqrt{t+t_0}} \right)^k \\
&  \leq &   K \left( \sup_{y \geq 1} \frac{w (y)}{y^k} \right) e^{-(x-2 (t+t_0))}  \left( x-2 (t+t_0) + r' \ln (t+t_0) \right)^k \\
&  \leq &   K \left( \sup_{y \geq 1} \frac{w (y)}{y^k} \right)   e^{-(\sqrt{t}-r\ln(t+t_0))}(\sqrt t +(r'-r)\ln(t+t_0))^k,
\end{eqnarray*}
where we have used the fact that $k = 1 -2r$. Since Lemma~\ref{lem:ode-w} ensures that $ \sup_{y \geq 1} \frac{w (y)}{y^k} < + \infty$, we conclude that
\begin{equation}\label{far way zone}
\lim_{t \to +\infty}\sup_{x \geq 2 (t+t_0) - r \ln (t+t_0) + \sqrt{t}} u(t,x) = 0.
\end{equation}

To obtain an upper estimate between the expected moving frame of the level sets, i.e. $2 t - r\ln t + \mathcal{O}(1)$, and $2 (t+t_0)- r \ln (t+t_0) + \sqrt{t}$ beyond which the above limit holds, the proof differs slightly between the delay and advance cases. We detail this in the next two steps.

\paragraph{Step 3: conclusion in the delay case $0\leq r<\frac32$.} Here, $r' = r$ by~\eqref{r'r} so that \eqref{comparisonsuper1} yields
$$
u (t,x) \leq K e^{-(x-2(t+t_0) + r \ln (t+t_0))
} \sqrt{t+t_0} \times w\left( \frac{x-2(t+t_0) + r \ln (t+t_0)}{\sqrt{t+t_0}} \right),
$$
for $x \geq 2(t+t_0) - r \ln (t+t_0) +1$, where $t_0 >1$ is fixed large enough. Recalling that $w(y) \sim C y^k$ with $k = 1-2r \leq 1$ as $y \to +\infty$, and also that $w (0) = 0$ and $w' (0)=1$, we have that $w(y)/y$ is bounded on $(0,+\infty)$. Then
\begin{equation*}
u (t,x) \leq  K \left(\sup _{y\geq 0} \frac{w(y)}{y}\right)   (x-2(t+t_0) + r \ln (t+t_0) ) e^{-(x-2(t+t_0) + r \ln (t+t_0))},
\end{equation*}
for $x \geq 2(t+t_0) - r \ln (t+t_0) +1$, from which the conclusion~\eqref{conclusion-super-nancy} of Proposition~\ref{prop:super-nancy} follows. Furthermore, evaluating the previous inequality at $x =2t - r \ln t +y$
with $t \geq 1$ and $y \geq 2t_0 +1$, we find some $k_1>0$ large enough so that
$$
u(t,2t- r \ln t +y) \leq k_1 (y+1) e^{-y} \quad \text{ for any $t \geq 1$, $y \geq 2t_0 +1$}.
$$
From the global boundedness of $u$ and the positivity of the right-hand term for $y \geq 0$, one then infers~\eqref{upper estimate front}, which concludes the proof of Proposition \ref{prop:super-nancy} in the delay case.

\paragraph{Step 4: conclusion in the advance case $r<0$.} Here, $r'=0 > r$. By \eqref{r'r}, the same computation as in Step 3 only provides an upper estimate on the location of the level sets of the form $2t - r \ln t + C_m \ln (\ln t)$. Therefore we must refine the argument. First, for $t\geq 1$, \eqref{comparisonsuper1} now yields
$$
u\left(t, 2(t+t_0) - r \ln (t+t_0) + \sqrt{t} \right) \leq   K e^{- \sqrt{t}} \sqrt{t+t_0}\times w \left(\frac{-r \ln (t+t_0) + \sqrt{t}}{\sqrt{t+t_0}} \right) ,$$
and thus there is $T\gg 1$ such that 
\begin{equation}\label{supersol_new_advance}
u\left(t ,2(t+t_0)-  r \ln (t+t_0) + \sqrt{t} \right) \leq  2 K \sqrt{t} e^{-\sqrt{t}}  w(1), \quad \forall t\geq T. 
\end{equation}
Besides, since 
$$\varphi (z) := z e^{-z}$$ satisfies $\varphi '' + 2 \varphi ' + \varphi = 0$ and $\varphi ' \leq 0$ on $(1,+\infty)$, and recalling that $r<0$ here, it is straightforward that, for any $K_2 > e$,
$$\overline{u}_2 (t,x) := \left\{ 
\begin{array}{ll}
1 & \ \mbox{ if } \ x < 2t - r \ln t  +1 , \vspace{3pt} \\
\min \{1, K_2 \,\varphi (x-2t + r \ln t )  \}  & \ \mbox{ if } \ x \geq 2t - r \ln t +1 ,\end{array}
\right.
$$ is a generalized super-solution of \eqref{eq}. Furthermore,
$$\overline{u}_2 \left(t, 2(t+t_0)-r \ln (t+t_0) + \sqrt{t} \right) \sim K_2 e^{-2t_0} \sqrt{t} e^{-\sqrt{t}}\quad \text{ as } t \to +\infty. 
$$
Putting this together with~\eqref{supersol_new_advance}, we get, up to increasing $T$ if necessary, that
$$
u\left (t, 2(t+t_0) -r \ln (t+t_0) + \sqrt{t}\right ) \leq \overline{u}_2 \left(t, 2(t+t_0)-r \ln (t+t_0) + \sqrt{t} \right), \quad \forall t \geq T,
$$
 provided that $K_2>2Kw(1) e^{2t_0}$. As in Step~2, due to $u \leq 1$, we may increase $K_2>0$ if necessary so that 
$$u(T,x) \leq \overline{u}_2 (T,x),$$
for all $x \leq 2(T+t_0) - r \ln (T+t_0) + \sqrt{T}$. Applying the parabolic comparison principle on
$$
\Omega_T:=\{ (t,x): t \geq T \, \mbox{ and } \, x \leq 2(t+t_0) - r \ln (t+t_0) + \sqrt{t} \},
$$
we end up with
$$
u(t,x) \leq \overline{u}_2 (t,x), \quad \forall (t,x)\in \Omega_T.
$$
The estimate \eqref{upper estimate front} follows immediately from the definition of $\overline u_2$, at least for $t\geq T^*$  with $T^*>1$ large enough and $1 \leq y \leq \frac{\sqrt{t}}{2}$ with $k_1 = K_2$. For $0 \leq y \leq 1$, or for $1 \leq t \leq T^*$ and $0 \leq y \leq \frac {\sqrt{t}} 2$, the same estimate holds up to increasing the factor~$k_1$, thanks to the boundedness of $u$. Moreover, putting the previous inequality together with~\eqref{far way zone}, since any $m$-level set of $\overline{u}_2$ is located around $2t - r \ln t + \mathcal{O}(1)$, we finally deduce~\eqref{conclusion-super-nancy}. \qed

\subsection{An alternate proof in the advance case ($k>1$)}\label{ss:alternate-proof-upper}

We start with an estimate on the solution $v(t,x)$ to \eqref{heat+v} (the linearized equation of \eqref{eq} around the invading steady state) starting from $u_0$, at a moving point $x=2t+\mathcal O(\sqrt{t})$. 

\begin{lemma}\label{lem:heat+v} Let $v=v(t,x)$ be the solution to
\begin{equation}
\label{heat+v}
 \partial _t v=\partial_{xx}  v+v, \quad t>0,\;  x\in \R,
\end{equation}
starting from $u_0$ satisfying Assumption \ref{ass:initial} with $k>1$. Let $c>0$ and $t_0>0$ be given. Then there is $c_0>0$ such that
$$
v(t,2t+c\sqrt t)\leq c_0 t^{\frac k 2} e^{-c\sqrt t}, \quad \forall t\geq t_0.
$$
\end{lemma}

\begin{proof} Denoting $G(t,x):=\frac{1}{\sqrt{4\pi t}}e^{-\frac{x^2}{4t}}$ the one dimensional Heat kernel, we  have
\begin{equation}\label{1}
e^{-t} v(t,x) =\int_\R G(t,x-y)u_0(y)\,dy\leq  \int _{-\infty}^1 G(t,x-y)\,dy+A \int_1 ^{+\infty} G(t,x-y)  y^ke^{-y}\, dy,
\end{equation}
from Assumption \ref{ass:initial}. Evaluating at $x=2t+c\sqrt t$, the first term in the above right hand side is nothing else than
\begin{equation}
\label{2}
\frac{1}{\sqrt \pi}\int _{-\infty}^{-\sqrt t -\frac c 2+\frac{1}{2\sqrt t}} e^{-z^2}\,dz\sim \frac{1}{2\sqrt \pi}\frac{1}{\sqrt t}e^{-(\sqrt t +\frac c 2-\frac{1}{2\sqrt t})^2}, \text{ as } t\to +\infty.
\end{equation}
As for the second term, using the change of variable $y=2\sqrt{t}z+x-2t$ with $x=2t+c\sqrt t$, some straightforward computations yield
\begin{equation}\label{3}
A\int_1 ^{+\infty} G(t,x-y)  y^ke^{-y}\, dy = A\frac{t^{\frac{k}{2}}e^{-(t+c\sqrt{t})}}{\sqrt{\pi}}\int^{+\infty}_{-\frac{c}{2}+\frac{1}{2\sqrt t}}(2z+c)^ke^{-z^2}\ dz \leq c' t^{\frac{k}{2}}e^{-(t+c\sqrt{t})},
 \end{equation}
 for some $c'>0$.  The combination of \eqref{1}, \eqref{2} and \eqref{3} concludes the proof (and so if we only assume $k>-1$ rather than $k>1$).
\end{proof}

The above upper estimate now allows us to use the minimal traveling wave, see Lemma~\ref{lem:tw-sub-super}, as an accurate super-solution.

\begin{proof}[Proof of Proposition \ref{prop:super-nancy} when $k>1$]  We start with some preparation. From Lemma \ref{lem:heat+v} (and the comparison principle), there is $c_0>0$ such that
\begin{equation*}
u(t,2t+\sqrt t)\leq c_0  t^{\frac k 2}e^{-\sqrt t}, \quad \forall t\geq 1.
\end{equation*}
 Since $U(z)\sim ze^{-z}$ as $z\to +\infty$, we can find $z_0>0$ such that $U(z)\geq \frac 1 2 ze^{-z}$ for all $z\geq z_0$. Select $T>1$ large enough such that
 $$
 \sqrt t+r\ln(1+t)\geq \frac 1 2 \sqrt t \geq z_0, \quad \forall t \geq T,
 $$
 and $K>1$ large enough such that
 \begin{equation}\label{22}
 K>4c_0, \; \text{ and } \;  KU(\sqrt T+r\ln (1+T))>1.
 \end{equation}
 
 We aim at applying the comparison principle in the region $\Omega_T:=\{(t,x)\in [T,+\infty)\times \R:  x\le 2t+\sqrt{t}\}$. We define
 $$
 \overline{u}(t,x):=KU(x-2t+r\ln (1+t)),
 $$ 
which is a super-solution since
\begin{eqnarray*}
 \partial _t \overline u -\partial_{xx}\overline u-\overline u(1-\overline u)&=&K\left(-2+\frac{r}{1+t}\right)U'-KU''-KU+K^2U^2\\
&=& K\frac{r}{1+T}U'+(K^2-K)U^2
\end{eqnarray*} 
is positive since $K>1$, $r<0$, $U'<0$. On the lateral boundary ($t\geq T$, $x=2t+\sqrt t$) of $\Omega_T$, we have
\begin{eqnarray*}
\overline u (t,2t+\sqrt t)&=& KU(\sqrt t +r\ln(1+t))\\
&\geq & \frac K 2 (\sqrt t+r\ln(1+t))e^{-\sqrt t -r\ln(1+t)}\\
&\geq & \frac K 4 \sqrt t e^{-\sqrt t} (1+t)^{-r}\geq \frac K 4 t^{\frac k 2} e^{-\sqrt t}\geq u(t,2t+\sqrt t),
\end{eqnarray*}
from \eqref{22}. On the boundary ($t=T$, $x\leq 2T+\sqrt T$) of $\Omega_T$ we have 
$$
\overline u (T,x)\geq KU(\sqrt T +r\ln(1+T))>1\geq u(T,x).
$$
We thus deduce from the comparison principle that 
$$
u(t,x)\leq KU(x-2t+r\ln (1+t)), \quad \text{ for all } (t,x)\in \Omega_T,
$$
from which the conclusion of  Proposition \ref{prop:super-nancy} easily follows. 
\end{proof}

\section{The lower estimate}\label{s:lower}

This section is devoted to the following lower estimate. Recall that $E_m$ denotes the $m$-level set of the solution, as defined in~\eqref{defi-level-set}.

\begin{proposition}[Lower estimate]\label{prop:lower}
Let $u=u(t,x)$ be the solution to \eqref{eq} starting from $u_0$ satisfying Assumption~\ref{ass:initial} with $k>-2$. Let $m\in(0,1)$ be given. Then, there are $C_m>0$ and $T_m>0$ such that
\begin{equation}
\label{conclusion-lower}
E_m(t)\subset (2t-r\ln t -C_m,+\infty), \quad \forall t\geq T_m,
\end{equation}
where
$$r =\frac{1-k}{2} < \frac{3}{2}.$$
Moreover, there are $k_2>0$ and $\sigma_2>0$ such that
\begin{equation}\label{low estimate front}
u(t,2t-r\ln t+y)\ge k_2ye^{-y}\quad \text{for any  $t\geq 1$, $0 \le y\le \sigma_2\sqrt{t}$}.
\end{equation}
\end{proposition}

\subsection{A lower bound in the far away zone}

We begin by demonstrating a lower estimate on the solution to \eqref{eq}, starting from $u_0$, evaluated at a moving point $x=2t+\mathcal O(\sqrt t)$. Note that this estimate aligns with a shift of the minimal front~$U$, so that its level sets are positioned around $x= 2t - r \ln t$. This will turn out to be crucial for bootstrapping our argument with a second sub-solution in subsection~\ref{sec: proof}.

\begin{lemma}\label{lem:lower_0}
Let $u= u(t,x)$ be the solution to \eqref{eq} starting from $u_0$ satisfying Assumption~\ref{ass:initial} with $k>-2$. Then there exists $\varepsilon >0$ such that
\begin{equation}\label{eq:lower_0}
u(t,2t+\sqrt{t}) \geq  \varepsilon e^{-\sqrt{t}} t^{\frac{1}{2} - r}, \quad \forall t\geq T_0,
\end{equation}
for some $T_0>1$ large enough.
\end{lemma}

\begin{proof} We proceed into two steps. The first one consists in turning our computations from Section~\ref{s:formal} into an actual sub-solution of \eqref{eq}. Next, we will prove that this sub-solution provides the lower estimate at $x =2t + \sqrt{t}$.

\paragraph{Step 1: the first sub-solution.} For $t>0$, $z\geq 0$, we define as before
 \begin{equation}\label{def:psi-bis}
\psi (t,z) :=  e^{-z} t^{\frac 1 2 + r' - r} \, w \left( \frac{z}{\sqrt{t}}\right),
\end{equation} 
where 
the function $w$ is given by Lemma~\ref{lem:ode-w}, and now
\begin{equation}\label{r'r-bis}
r ' = r - 2 < 0.
\end{equation}
For $t_0>1$, $M>0$, $\ep>0$ to be specified later, let us prove that 
\begin{equation}
\label{subsol}
\underline  u(t,x):=\left(1+\frac{M}{\sqrt t}\right)\ep \,\psi(t,x-2t+r' \ln t)
\end{equation}
is a sub-solution of \eqref{eq} for all $t \geq t_0$ and $x \geq 2t - r' \ln t$. By construction $\underline  u(t,x)$ also satisfies the Dirichlet boundary condition
$$\underline{u} (t, 2t - r' \ln t) = 0.$$
We use the computations of Step~1 in subsection~\ref{ss:proof-upper} again and get, denoting $z=x-2t+r'\ln t$, 
$$
\begin{array}{lll}
\mathcal L \underline u (t,x) &= & \partial_t \underline {u} - \partial _{xx} \underline u -\underline u +\underline u ^2\vspace{7pt}\\
& =&  \ep\frac{e^{-z}}{t^{1+r-r'}}\Bigg\{ -\frac M 2 w\left(\frac{z}{\sqrt t }\right)\\
&&\qquad \qquad\qquad+\left(1+\frac{M}{\sqrt t}\right)\left[r'w'\left(\frac{z}{\sqrt t }\right)+\left(1+\frac{M}{\sqrt t}\right)\ep e^{-z}t^{2+r'-r}w^2\left(\frac{z}{\sqrt t }\right)\right]\Bigg\}\vspace{7pt}\\
& = &  \ep\frac{e^{-z}}{t^{3}}\Bigg\{ -\frac M 2 w\left(\frac{z}{\sqrt t }\right)+\left(1+\frac{M}{\sqrt t}\right)\left[r'w'\left(\frac{z}{\sqrt t }\right)+\left(1+\frac{M}{\sqrt t}\right)\ep e^{-z} w^2\left(\frac{z}{\sqrt t }\right)\right]\Bigg\},
\end{array}
$$
where on the last line we simply used $r' = r-2$. 

Next, we distinguish two subdomains. First, we use the fact that $w' (0) =1$, hence there exists $\delta >0$ such that $w' (y)> \frac{1}{2}$ in $y \in [0,\delta]$. Then, due to $r' <0$, we have for any $t \geq 1$ and $z = x -2t + r' \ln t \in (0,\delta  \sqrt{t}]$,
\begin{eqnarray*}
\mathcal L \underline u (t,x)  \leq \ep\frac{e^{-z}}{t^{3}}\left(1+\frac{M}{\sqrt t}\right)\left[ \frac{r'}{2} +\left(1+ M\right)\ep \max_{0\leq y \leq \delta} w^2(y) \right],
\end{eqnarray*}
which is negative for any 
\begin{equation}\label{eps_sub1}
\varepsilon < \frac{-r'}{2 (1+M) \max_{0\leq y \leq \delta} w^2(y) }.
\end{equation}
Furthermore, from Lemma \ref{lem:ode-w}, we can choose $M\gg 1$ such that, for $y \geq \delta$,
$$\frac{8 |r' w' (y)|}{w (y)} \leq M , \qquad w (y)  \leq M y^k .$$
Then, for $t_0$ such that
$$
1 + \frac{M}{\sqrt{t_0}} \leq 2,
$$
we get,  for any $t\geq t_0$ and any $z = x - 2t + r' \ln t \geq \delta \sqrt{t}$, that
\begin{eqnarray*}
\mathcal L \underline u (t,x) & = & \ep\frac{e^{-z}}{t^{3}}\left\{ -\frac M 2 w\left(\frac{z}{\sqrt t }\right)+\left(1+\frac{M}{\sqrt t}\right)\left[r'w'\left(\frac{z}{\sqrt t }\right)+\left(1+\frac{M}{\sqrt t}\right)\ep e^{-z} w^2\left(\frac{z}{\sqrt t }\right)\right]\right\}\\
& \leq &  \ep\frac{e^{-z}}{t^{3}} w\left(\frac{z}{\sqrt t }\right) \, \left\{ -\frac M 4+ 4 M \ep e^{-z} \frac{z^{k}}{t^{k/2 }}\right\}.
\end{eqnarray*}
Up to increasing $t_0$ if necessary, we have for any $t \geq t_0$ that $\max _{z\geq \delta \sqrt t}e^{-z}z^k=e^{-\delta \sqrt t} (\delta \sqrt t)^k$, and then (up to also reducing $\ep>0$),

$$
\mathcal L \underline u (t,x)  \leq   \ep\frac{e^{-z}}{t^{3}} w\left(\frac{z}{\sqrt t }\right) \, \left\{ -\frac M 4+ 4 M \ep e^{-\delta \sqrt{t} } \delta^k \right\}<0,
$$
for any $x \geq 2t - r' \ln t + \delta \sqrt{t}$.

Putting the inequalities in both subdomains together, we found $M>0$ such that, for any $\varepsilon >0$ and $t_0 >1$ respectively small and large enough, the function~\eqref{subsol} is a sub-solution of~\eqref{eq} for $t \geq t_0$ and $x \geq 2t -r' \ln t$.

\paragraph{Step 2: comparison with the solution.} We claim that, up to reducing $\varepsilon >0$ and increasing $t_0 >1$, we have that
\begin{equation}\label{step2_sub1_toprove}
u_0 \geq \underline{u} (t_0, \cdot),
\end{equation}
where $\underline{u}$ is defined in \eqref{subsol}. Recall that $\underline{u} (t, x)$ is supported in $\{ x \geq 2 t - r' \ln t \}$, and hereafter we extend it continuously by~$0$ on $\{x < 2 t - r' \ln t \}$. Therefore, by Assumption~\ref{ass:initial} and up to increasing $t_0$, we only need to show that
$$
a x^k e^{-x} \geq  \underline{u} (t_0, x), \quad \forall x \geq 2t_0 -r' \ln t_0 \geq 1,
$$ 
or equivalently,
\begin{equation}\label{33}
a x^k   \geq C(t_0) \varepsilon  w\left( \frac{x-2t_0 +r' \ln t_0}{\sqrt{t_0}}\right), \quad \forall x \geq 2t_0 -r' \ln t_0 \geq 1,
\end{equation}
where $C (t_0) := \left(1 + \frac{M}{\sqrt{t_0}} \right) e^{2t_0} t_0^{\frac{1}{2} - r}$. It follows from Lemma~\ref{lem:ode-w}, more precisely~\eqref{w-infinity}, that \eqref{33} is satisfied if~$\varepsilon$ is small enough.

We conclude that \eqref{step2_sub1_toprove} holds. Then, by the parabolic comparison principle,
\begin{equation*}
u (t,x) \geq \underline{u} (t+t_0 , x),
\end{equation*}
for any $t>0$ and $x \in \mathbb{R}$. One can verify that
$\|\underline{u} (t+t_0 , \cdot)\|_{L^{\infty} (\R)}$ decreases
algebraically to $0$ as $t \to +\infty$, due to $r' = r -2$. In particular, we cannot yet conclude anything on the large-time position of the level sets. However, this decay in time, and specifically the choice of $r'$, ensured that the nonlinear component of the reaction term was small enough when evaluating $\mathcal{L} \underline{u}$. Therefore, it played a crucial role in~$\underline{u}$ being a sub-solution.

Still, we now know that, for any $t>0$ and $x =2 t + \sqrt{t}$,
$$
u (t, 2t + \sqrt{t}) \geq  \underline{u} (t+t_0, 2t + \sqrt{t}) \geq \varepsilon e^{-\sqrt{t}}  e^{2t_0} (t+t_0)^{\frac{1}{2}  - r} w \left( \frac{\sqrt{t}-2t_0 + r' \ln (t+t_0)}{\sqrt{t+t_0}} \right), 
$$
and then
$$
u (t, 2t + \sqrt{t}) \geq\frac 1  2 \ep e^{2t_0} w(1) e^{-\sqrt t} t^{\frac 12 -r},
$$
provided $t$ is large enough. This concludes the proof of \eqref{eq:lower_0}. 
\end{proof}

\subsection{Proof of Proposition~\ref{prop:lower} }\label{sec: proof}

We now address how the lower estimate on the level sets of the solution to \eqref{eq} follows from Lemma~\ref{lem:lower_0}. The main idea is to construct a second sub-solution on a left half-line, which turns out to be more challenging in the advance case.

\paragraph{Step 3: the second sub-solution in the delay case $0 \leq r < \frac32$.} In this case, the minimal traveling wave itself serves as a sub-solution, see Lemma \ref{lem:tw-sub-super}. Denote by $U$ the solution to
$$\left\{
\begin{array}{l}
U '' + 2 U ' + U(1-U)=0 \quad \text{ on } \R, \vspace{3pt}\\
U (-\infty) = 1, \quad U(0)=\frac 12,  \quad U (+\infty)=0,
\end{array}
\right.
$$
which is known to satisfy $U'<0$ and $U(z)\sim Bze^{-z}$ as $z\to +\infty$, for some $B>0$. Then, for any $0<\alpha <1$ to be selected below, and due to $r \geq 0$, it is straightforward to check that
$$
\underline{u}_2 (t,x) := \alpha U (x-2t + r \ln (1+t))
$$
is a sub-solution of \eqref{eq} for $t\geq 0$ and $x\in \R$.

Next, we aim at applying the parabolic comparison principle on $\{(t,x): t\geq T, \, x\leq 2t+\sqrt t\}$ for some appropriate $T>T_0$, where $T_0$ is chosen as in Lemma \ref{lem:lower_0}. On the one hand, on the lateral boundary $x=2t+\sqrt t$, we have
$$
\underline{u}_2 (t,2t+\sqrt{t}) =\alpha U(\sqrt t+r\ln(1+t)) \sim \alpha B t^{\frac{1}{2}-r} e^{-\sqrt{t}}, \quad \text{ as } t\to +\infty. 
$$
In view of the lower estimate \eqref{eq:lower_0} of the solution $u$, it follows that, by setting $0<\alpha<\frac{\ep}{2B}$, we have, for some $T>T_0$ large enough, that
$$
u (t,2t+ \sqrt{t}) \geq \underline{u}_2 (t, 2t + \sqrt{t}),\quad \forall t\geq T.
$$
On the other hand, since spreading occurs (at speed $c=2$)  we have $\inf _{x\leq 0} u(T,x)\geq \frac 12$, up to increasing $T$ again. Such a time~$T$ now being fixed, and since $u (t,\cdot) >0$ for any $t >0$ by the strong maximum principle, we get that 
$$\inf _{x\leq 2T+\sqrt T} u(T,x)>0.$$
Hence, imposing $0<\alpha< \inf _{x\leq 2T+\sqrt T} u(T,x) $, we immediately get that $u(T,x)\geq \underline u_2(T,x)$ for all $x \leq 2T+\sqrt T$.

As a result, for any $0<\alpha< \min \{ \frac{\varepsilon}{2B}, \inf _{x\leq 2T+\sqrt T} u(T,x) \}$, there exists $T>0$ such that we can deduce from the comparison principle that
\begin{equation}\label{lower_delay_end}
u(t,x)\geq \alpha U(x-2t+r\ln(1+t)), \quad \forall t\geq T,\; \forall x\leq 2t+\sqrt t .
\end{equation}
Due to the asymptotics of $U$ as $y \to +\infty$, there exists some $k_2 >0$ small enough such that
$$
U(y) \geq \frac{2 k_2}{\alpha} y e^{-y}\quad \text{ for any }y \geq 0.
$$
Thus, from~\eqref{lower_delay_end}, we get
$$u(t,2t- r \ln t + y ) \geq \alpha U (y + r \ln (1+ 1/t)) \geq k_2 y e^{-y},$$
for any $t$ large enough and $0 \leq y \leq \sqrt{t} + r \ln t $. We infer that~\eqref{low estimate front} holds.

Moreover, \eqref{conclusion-lower} is a direct consequence of~\eqref{lower_delay_end} for \lq\lq small'' level sets, say $0<m<m_0:=\frac \alpha 2$, and then for all level sets $0<m<1$ thanks to the following result (which does not require $r\geq 0$ and will be used again later).

\begin{lemma}\label{lem:small-level-sets} Let $0<m_0<1$ be given. Then, if~\eqref{conclusion-lower} holds for any $0<m<m_0$, it actually holds for any $0<m<1$.
\end{lemma}

\begin{proof}
Let us consider the remaining case $m_0\leq m<1$. Let us denote by $v=v(t,x)$ the solution to \eqref{eq} starting from 
 \begin{equation}\label{initial-data-v}
v_0(x):= \begin{cases}\frac{ m_0}2 &\text{
if } x\leq -1 ,
\\
-\frac{m_0}2 x &\text{ if } -1<x<0 ,\\
0 &\text{ if  }x\geq 0.
\end{cases}
\end{equation}
Since spreading occurs (at speed $c=2$), there is a time $\tau^*=\tau^*_{m_0,m}>0$ such that
\begin{equation}
\label{v-grand}
 v(\tau^*,x)>m, \quad \forall x\leq 0.
 \end{equation}
From~\eqref{conclusion-lower} with~$m= \frac{3}{4}m_0$, we know that there are $C_0>0$ and $T_0>\max \{ 1,\tau^* \}$ such that 
$$
u(t,x) \geq \frac 3 4 m_0, \quad \forall t\geq T_0,\; \forall x\leq x^*(t):=2t-r\ln t -C_0. 
$$
In view of definition \eqref{initial-data-v}, this implies
 $$
 u(t,x)\geq v_0(x-x^*(t)),\quad \forall t\geq T_0,\; \forall x\in \R,
 $$
 so that the comparison principle yields
 $$
 u(t+\tau,x)\geq v(\tau,x-x^*(t)),\quad \forall t\geq T_0,\; \forall \tau \geq 0,\; \forall x\in \R.
 $$
Using~\eqref{v-grand}, we get that
 $$
 u(t+\tau^*,x)>m,\quad \forall t \geq T_0,\; \forall x\leq x^* (t),
 $$
 which proves the result.
 \end{proof}
 
\paragraph{Step 4: the second sub-solution in the advance case $r < 0$.} We look for a sub-solution that is valid for $t >0$ and $x \leq 2t + \sqrt{t}$. In the advance case, the traveling wave cannot be turned into an appropriate sub-solution. Instead, an adequate tool for the construction is the following ODE solution.

\begin{lemma}\label{lem_yet}
Let $\gamma > 1$ be given. Denote $\varphi=\varphi_\gamma$ the solution of the ODE Cauchy problem
$$\left\{
\begin{array}{l}
\varphi '' + 2 \varphi ' + \varphi - \gamma \varphi^2  = 0 \quad \text{ on } (0,+\infty), \vspace{3pt}\\
\varphi  (0) = \frac{1}{2 \gamma},\vspace{3pt}\\
 \varphi ' (0) = 0.
\end{array}
\right.
$$
Then 
$$
\varphi' (z) < 0, \qquad \frac{\varphi '(z)}{\varphi(z)}  \geq -1, \quad \forall z>0.
$$
Moreover, there exists $B=B_\gamma >0$ such that
\begin{equation}\label{asympto-B}
\varphi (z) \sim B  z e^{-z}, \quad \text{ as }  z\to +\infty.
\end{equation}
\end{lemma}

\begin{proof} This is rather standard from the same ODE techniques which have been extensively used to characterize traveling waves of reaction-diffusion equations, but we will sketch the proof for the sake of completeness. 
First, the ODE is recast 
$$\left\{ \begin{array}{l}
p' = q, \\
q' = - 2 q - p + \gamma p^2,\end{array}
 \right.$$
 and we use a phase plane analysis. The trajectory corresponding to $\varphi$ in the $(p,q)$ plane starts from $(\frac{1}{2\gamma},0)$ and immediately enters the subset $\{ 0 < p < \frac{1}{\gamma}, \, q < 0\}$. Then it cannot cross the horizontal segment $\{0<p< \frac{1}{\gamma}, \, q=0\}$, nor the vertical half-line $\{p= \frac{1}{\gamma}, \, q< 0\}$. It also cannot cross the diagonal $\{ q = -p \}$ (if so then, at the first crossing point, we would have $-1\leq \frac{q'}{p'}=\frac{-q+\gamma p^2}{q}=-1+\gamma \frac{p^2}{q}$, which is impossible since $q<0$). It already follows that $\varphi' < 0$ and $\varphi' / \varphi \geq -1$, and that $\varphi (+\infty) = 0$. By standard ODE perturbation theory, we also get that
$$
\varphi (z) \sim (Bz +C) e^{-z}, \quad \text{ as } z\to +\infty,
$$
for some $B\geq 0$. We assume by contradiction that $B=0$. Since the PDE $\partial_t u = \partial_{xx} u + u - \gamma u^2$
is still of the KPP type, it admits a minimal traveling wave $U=U_\gamma$ solving
$$\left\{
\begin{array}{l}
U '' + 2 U ' + U-\gamma U^2=0 \quad \text{ on } \R, \vspace{3pt}\\
U (-\infty) = \frac 1\gamma, \quad U(0)=\frac 1{2\gamma},  \quad U (+\infty)=0,
\end{array}
\right.
$$
which is known to satisfy $U'<0$ and $U(z)\sim  Aze^{-z}$ as $z\to +\infty$, for some $A>0$. Hence~$\varphi$ and~$U$ solve the same ODE and, from the behaviors at $z=0$ and $z\to +\infty$, the trajectories in the phase plane should intersect, which contradicts the Cauchy-Lipschitz theorem. We conclude that $B>0$, and~\eqref{asympto-B} holds.
\end{proof}

Equipped with this, we now claim that we can find $\eta>0$ and $\gamma >1$ such that
\begin{equation}\label{sub_last?}
\underline{u}_2 (t,x) := \left\{ 
\begin{array}{ll}
e^{\eta \frac{z}{\sqrt{t}} } \varphi (z) & \mbox{ if } z >0, \vspace{3pt} \\
  \frac{1}{2 \gamma } & \mbox{ if } z \leq 0 , 
\end{array} \quad \text{ where }  z:=x -2t + r \ln t,
\right.
\end{equation}
is a sub-solution of~\eqref{eq} for $t > 0$, $z \leq 2 \sqrt{t}$. The function $\underline{u}_2$ is clearly a sub-solution for $z < 0$, it is continuous and has the good angle at the gluing point in the sense that
$$
\partial _x \underline u _2(t, (2t-r\ln t)^+)>0=\partial _x \underline u _2(t, (2t-r\ln t)^-), \quad \forall t > 0,
$$
since $\varphi (0) =\frac{1}{2 \gamma}$ and $\varphi ' (0) = 0$. Thus we only need to deal with the differential inequality when $0<z\leq 2\sqrt t$. By using Lemma~\ref{lem_yet} and the ODE satisfied by $\varphi$,  we compute
\begin{eqnarray*}
e^{-\eta \frac{z}{\sqrt t}}\mathcal L \underline u_2 (t,x) &= & e^{-\eta \frac{z}{\sqrt t}}\left(\partial_t \underline {u}_2 - \partial _{xx} \underline u_2 -\underline u_2 +\underline u_2 ^2\right)\\
& = & - \frac{\eta z}{2 t^{3/2}}  \varphi (z)  -  \frac{2 \eta}{\sqrt{t}}  \varphi ' (z) - \frac{\eta^2}{t} \varphi (z) -  \varphi '' (z)\\
& & \hspace{5mm} - \left( 2 - \frac{r}{t} \right) \left( \frac{\eta}{\sqrt{t}} \varphi (z) + \varphi ' (z) \right) -\varphi (z)  + e^{\eta \frac{z}{\sqrt{t}} } \varphi ^2(z) \\
& = &  - \frac{\eta z}{2 t^{3/2}} \varphi (z)  - \frac{2 \eta}{\sqrt{t}}( \varphi ' (z) + \varphi (z)) -  \frac{\eta^2}{t} \varphi (z) + \frac{r}{t} \varphi ' (z)  + \frac{r\eta}{t^{3/2}} \varphi (z)  \\
& & \hspace{5mm}  + \left(  e^{\eta\frac{z}{\sqrt{t}} }- \gamma  \right) \varphi^2 (z).
\end{eqnarray*}
Recalling that $r<0$, $z>0$, $\varphi>0$, $\varphi'<0$ and $\varphi'+\varphi \geq 0$, we get rid of the some nonpositive terms and find that
\begin{eqnarray*}
\frac{e^{-\eta \frac{z}{\sqrt t}}}{\varphi(z)} \mathcal L \underline u_2 (t,x) & \leq & \frac{ -  \eta^2 - r} t    +    \left(  e^{\eta\frac{z}{\sqrt{t}} }  - \gamma \right) \varphi (z).
\end{eqnarray*}
Picking $\eta = \sqrt{-r} >0$ and $\gamma = e^{2 \eta} >1$, and using $z \leq 2 \sqrt{t}$, we get
\begin{eqnarray*}
\frac{e^{-\eta \frac{z}{\sqrt t}}}{\varphi(z)} \mathcal L \underline u_2 (t,x) & \leq &  (e^{2\eta}-\gamma)\varphi(z)=0.
\end{eqnarray*}
As a result, $\underline{u}_2$ is a sub-solution of \eqref{eq} for $t > 0$, $x  \leq 2t - r \ln t + 2\sqrt t$. 

Now, the conclusion is rather similar to that in the delay case. Precisely, we fix $\eta>0$, $\gamma>1$ as above. Recall that $B>0$ is set from \eqref{asympto-B} and $\ep>0$ is set from Lemma \ref{lem:lower_0}. Let $0<\alpha< 1$ be any small enough constant so that 
$$\alpha  e^{\eta} B<\frac \ep 2 . $$
Observe first that $\alpha \underline u _2$ is also  a sub-solution for $t\geq 0$, $x \leq 2t - r \ln t +  2\sqrt  t$. We aim at applying the parabolic comparison principle on $\{(t,x): t\geq T, x\leq 2t+\sqrt t\}$ for some appropriate $T>T_0$, where $T_0$ is defined as in Lemma~\ref{lem:lower_0}. On the lateral boundary $x=2t+\sqrt t$, we have
$$
\alpha \underline{u}_2 (t,2t+\sqrt{t}) =\alpha e^{\eta(1+r\frac{\ln t}{\sqrt t})}\varphi(\sqrt t+r\ln t)\sim  \alpha e^{\eta} B t^{\frac{1}{2}-r} e^{-\sqrt{t}}, \quad \text{ as } t\to +\infty. 
$$
In view of the lower estimate \eqref{eq:lower_0} of the solution $u$ and our choice of $\alpha$, we can find some $T>T_0$ large enough such that
$$
u (t,2t+ \sqrt{t}) \geq \alpha \underline{u}_2 (t, 2t + \sqrt{t}),\quad \forall t\geq T.
$$
Moreover, since spreading occurs (at speed $c=2$), up to increasing $T$ (that we now fix), we have $\inf _{x\leq 0} u(T,x)\geq \frac 12$, and then $\inf _{x\leq 2T+\sqrt T} u(T,x)>0$ thanks to the positivity of $u$. Hence,  by setting further  $0<\alpha<2\gamma e^{-2\eta}\inf _{x\leq 2T+\sqrt T} u(T,x)$, one may check that $u(T,x)\geq \alpha \underline u_2(T,x)$ for all $x\leq 2T+\sqrt T$.

As a result, we deduce from the comparison principle that
\begin{equation}\label{lower_advance_end}
u(t,x)\geq \alpha \underline u_2(t,x), \quad \forall t\geq T,\; \forall x\leq 2t+\sqrt t .
\end{equation}
The end of the proof proceeds as in the delay case. On the one hand, \eqref{conclusion-lower} immediately follows from~\eqref{lower_advance_end} for \lq\lq small'' level sets, and then for all level sets thanks to Lemma \ref{lem:small-level-sets}. On the other hand, \eqref{low estimate front} follows from~\eqref{lower_advance_end} together with the definition of $\underline{u}_2$ in~\eqref{sub_last?} and the asymptotics of~$\varphi$ by Lemma~\ref{lem_yet}. We omit the details. \qed

\section{The critical case $k=-2$}\label{s:critical}

In this section, we prove the following proposition, from which Theorem~\ref{th:critical} will follow.

\begin{proposition}[Level sets in the critical case]\label{had to be done}
Let $u=u(t,x)$ be the solution to \eqref{eq} starting from $u_0$ satisfying Assumption \ref{ass:initial} with $k=-2$. Let $m\in(0,1)$ be given. Then, there are $C_m>0$ and $T_m>0$ such that
\begin{equation}
\label{location-critical-case}
E_m (t) \subset \left(2t - \frac{3}{2} \ln t + \ln \ln t - C_m ,  2t - \frac{3}{2} \ln t +  \ln \ln t +C_m\right), \quad \forall t\geq T_m.
\end{equation}
Moreover, there are $T>0$, $k_1 > k_2 >0$ and $\sigma_1>0$ such that
\begin{equation}\label{critical estimate front}
k_2 y e^{-y} \leq u \Big( t,2t- \frac{3}{2}\ln t + \ln \ln t +y \Big)\le k_1 (y+1) e^{-y},
\end{equation}
for any $t \geq T$ and $0 \leq y \leq \sigma_1 \ln t$.
\end{proposition}

\paragraph{Step 1. Preliminary computations.} Let us first consider $v=v(t,x)$ the solution to the heat equation on the half-line with Dirichlet boundary condition, namely
$$
\left\{ \begin{array}{ll}
\partial_t v = \partial_{xx} v, &  \quad t>0, \ x >0, \vspace{3pt}\\
v (t,0) = 0, & \quad t>0 ,
\end{array}
\right.
$$
starting from
\begin{equation}\label{v-zero}
v _0(x) = 
\left\{
\begin{array}{ll}
1 & \mbox{ if } 0<x \leq 1 , \vspace{3pt}\\
\frac{1}{x^2} & \mbox{ if } x > 1.
\end{array}
\right.
\end{equation}
In particular,the function $e^{-(x-2t)} v (t,x-2t)$ solves the linearized equation $\partial_t u = \partial_{xx} u + u$ on a right half-line with a moving Dirichlet condition at $x = 2t$, and its initial value is also consistent with Assumption~\ref{ass:initial} with $k =-2$. Therefore, it serves as a potential candidate to accurately determine the position of the level sets.

First, we know that
\begin{eqnarray*}
v (t,x) &=& \frac{1}{\sqrt{4 \pi t}} \int_0^{+\infty} \left( e^{-\frac{(x-y)^2}{4t}} - e^{-\frac{(x+y)^2}{4t}} \right) v_0 (y) dy \\
&=& \frac{e^{-\frac{x^2}{4t}} }{\sqrt{4 \pi t}} \int_0^{+\infty} e^{-\frac{y^2}{4t} } \left( e^{\frac{yx}{2t}} - e^{-\frac{yx}{2t} } \right) v_0 (y) dy,
\end{eqnarray*}
so that
\begin{eqnarray*}
v (t, 2 \sqrt{t}) & = & \frac{e^{-1}}{\sqrt{4 \pi t}} \int_0^{+\infty}  e^{-\frac{y^2}{4t} } \left( e^{\frac{y}{\sqrt{t}}} - e^{-\frac{y}{\sqrt{t}} } \right) v_0 (y) dy \\
& = & \frac{e^{-1}}{\sqrt{4 \pi}} \int_0^{+\infty}  e^{-\frac{y^2}{4} } \left( e^{y} - e^{-y} \right) v_0 (\sqrt{t}y) dy \\
& = & \frac{e^{-1}}{\sqrt{4 \pi}} \int_0^{1/ \sqrt{t}} e^{-\frac{y^2}{4} } \left( e^{y} - e^{-y} \right)  dy + \frac{e^{-1}}{\sqrt{4 \pi}} \int_{1/ \sqrt{t}}^{+\infty}  e^{-\frac{y^2}{4} } \left( e^{y} - e^{-y} \right) \frac{1}{t y^2} dy \\
& =: & I_1 + I_2 .
\end{eqnarray*}
On the one hand, as $t \to +\infty$,
$$
I_1  \sim  \frac{e^{-1}}{\sqrt{ \pi}} \int_0^{1/ \sqrt{t}}   y   dy   =   \mathcal{O} \left( \frac{1}{t} \right).
$$
On the other hand, due to the integrand in $I_2$ being integrable at $+\infty$, but not at $0$, we have, as $t\to +\infty$, 
$$
I_2  \sim \frac{e^{-1}}{\sqrt{4 \pi}} \int_{1/ \sqrt{t}}^1  e^{-\frac{y^2}{4} } \left( e^{y} - e^{-y} \right) \frac{1}{t y^2} dy  \sim   \frac{e^{-1}}{\sqrt{\pi}} \int_{1/ \sqrt{t}}^1   \frac{1}{t y} dy  \sim    \frac{e^{-1}}{\sqrt{\pi}} \frac{\ln t }{2 t}.
$$
Therefore 
\begin{equation}\label{equiv-v-loin}
v (t, 2\sqrt{t}) \sim  \frac{e^{-1}}{\sqrt{\pi}} \frac{\ln t }{2 t},\quad \text{ as } t \to +\infty.
\end{equation}

Let us conclude this step by a formal discussion. If the ansatz $e^{-(x-2t)}v(t,x-2t)$ is accurate, we may expect $u (t, 2t + 2 \sqrt{t})$ to be well approximated by
$$e^{-2 \sqrt{t}} v(t,2 \sqrt{t}) \sim \frac{e^{-1}}{\sqrt \pi} e^{-2\sqrt{t}}\, \frac{\ln t}{2t},\quad \text{ as } t \to +\infty.$$
On the other hand, we have that
\begin{eqnarray*}
U \left(2t + 2\sqrt{t} - \Big( 2t - \frac{3}{2}\ln t + \ln \ln t \Big) \right) & = &
U \left( 2 \sqrt{t}+ \frac{3}{2} \ln t - \ln \ln t \right) \\
&  \sim & B  \times \left(2 \sqrt{t} + \frac{3}{2}\ln t - \ln \ln t \right) e^{-2 \sqrt{t} - \frac{3}{2}\ln t + \ln \ln t} \\
& \sim & 2 B e^{-2 \sqrt{t}} \, \frac{\ln t}{t}, \quad \text{ as } t \to +\infty.
\end{eqnarray*}
Observing that both estimates coincide, this suggests that the position of the front of the solution~$u$ should be around $2 t - \frac{3}{2} \ln t + \ln \ln t$, in agreement with our Theorem \ref{th:critical}.

\paragraph{Step 2. Lower estimate.} The rough plan is to first perturb the function 
$v$ defined in the previous step to make it a sub-solution, thereby obtaining a lower bound on the solution at $x=2t + 2 \sqrt{t}$. Subsequently, the lower estimate on the level sets will follow similarly to the noncritical delay case, thanks to the fact that a slowed-down minimal front also serves as a sub-solution.

Precisely, we consider
$$
\underline{u} (t,x) =  \delta \times \left( 1 + \frac{M}{(t+1)^{1/4}} \right) \times e^{-(x-2t)} v(t, x -2t), \quad t\geq 0, x\geq 2t,
$$
where $\delta>0$ and $M>0$ will be specified below, chosen to be respectively small and large enough. Notice that we multiplied the ansatz $e^{-(x-2t)} v(t,x-2t)$ by a decreasing in time function. As in the previous section, provided that~$v$ decays to $0$ fast enough, the resulting additional term when evaluating the KPP equation will absorb the nonlinear quadratic term that our computation in Step 1  had previously ignored.

First, by Assumption~\ref{ass:initial} (with $k=-2$) and our choice of $v (0,\cdot)$, it is straightforward to check that, for any $M>0$, there exists $\delta >0$ small enough so that
$$
\underline{u} (0, x )  \leq u_0 (x), \quad \forall x\geq 0.
$$
By the strong maximum principle, we also have that
$$\underline{u} (t,2t)=0<u(t,2t), \quad\forall  t>0.$$
Next, let us show that~$\underline{u}$ is a sub-solution. We  compute, for any $t>0$ and $x > 2t$,
\begin{eqnarray*}
&& \delta^{-1} e^{x-2t} \times \left( \partial_t \underline{u} - \partial_{xx} \underline{u} - \underline{u} + \underline{u}^2 \right) \\
&& \qquad =   v(t,x-2t) \times \left[ \delta \left(1+ \frac{M}{(t+1)^{1/4}}\right)^2 e^{-(x-2t)} v(t,x-2t) - \frac{M}{4 (t+1)^{5/4}}   \right]\\
&& \qquad \leq   v(t,x-2t) \times  \left[ e^{-(x-2t)}  v(t,x-2t) - \frac{M}{4 (t+1)^{5/4}}  \right],
\end{eqnarray*} 
assuming further $\delta(1+M)^2<1$. From the following lemma, whose proof is postponed, we can choose $M>0$ sufficiently large (and then $\delta>0$ sufficiently small) to ensure that $\underline u$ is a sub-solution.

\begin{lemma}\label{lem:critical_sub}
For any $\ep>0$, the function~$v$ defined in Step 1 satisfies 
$$
\sup_{ t >0,\, x >0} e^{-x} v (t,x) \times (t+1)^{\frac 3 2 -\ep} < + \infty.
$$
\end{lemma}

Hence, by the parabolic comparison principle, we  have that
$$
u(t,x) \geq \underline{u} (t,x)  = \delta  \left( 1 + \frac{M}{(t+1)^{1/4}}\right) e^{-(x-2t)} v(t, x -2t), \quad \text{for all } t \geq 0, x \geq 2t.
$$
As a result, we deduce from \eqref{equiv-v-loin} that  that there exists $\varepsilon >0$ such that
\begin{equation}\label{u-loin}
u(t,2t+2 \sqrt{t}) \geq \varepsilon e^{-2 \sqrt{t}} \times \frac{\ln t}{ t}, \quad \forall t>1.
\end{equation}

The rest of the proof proceeds similarly to the noncritical case, see Step~3 of subsection~\ref{sec: proof} for further details. Consider the second sub-solution (as can be easily verified)
$$\underline{u}_2 (t,x) := \alpha U \Big( x-2t + \frac{3}{2} \ln t - \ln \ln t \Big),$$
where $\alpha \in (0,1) $ and $U$ denotes the minimal travelling wave. Recalling $U(z)\sim Bze^{-z}$ as $z\to +\infty$, we have
$$
\underline{u}_2(t,2t+2\sqrt t) \sim 2\alpha B e^{-2\sqrt t}\times \frac{\ln t}{t}, \quad \text{ as } t\to +\infty.
$$
In view of \eqref{u-loin}, we may thus choose $\alpha$  small enough so that, for some large enough $T>0$,
$$
\underline{u}_2 (t,2t + 2 \sqrt{t})  \leq u (t, 2t + 2 \sqrt{t}),\quad \forall t\geq T.
$$
 Up to further reducing $\alpha$, we may also have
$$\underline{u}_2 (T, x) \leq u (T,x),$$
for $x \leq 2T + 2 \sqrt{T}$. Finally, by another application of the parabolic comparison principle, we have that
\begin{equation}\label{eq:subsol_almostdone}
\underline{u}_2 (t,x) \leq u (t,x), \quad \text{ for any $t \geq T$ and $x \leq 2t + 2 \sqrt{t}$}.
\end{equation}

We can now conclude that there exist $C_m, T_m >0$
$$E_m (t) \subset \left( 2t - \frac{3}{2} \ln t + \ln \ln t - C_m , +\infty \right) , \text{ for all } t\geq T_m,$$
first for small $0 < m < \frac{\alpha}{2}$, then for all $0 < m < 1$ by Lemma~\ref{lem:small-level-sets}. Let us point out that this already proves the drift is not exactly logarithmic. 

It also follows from \eqref{eq:subsol_almostdone} and the spatial asymptotics of $U$ (hence, of $\underline{u}_2$) that there exists $k_2 >0$ such that
\begin{equation*}
u \left(t, 2t - \frac{3}{2}\ln t + \ln \ln t + y\right) \geq k_2 y e^{-y},
\end{equation*}
for all $t \geq T$ and $0 \leq y \leq \sqrt{t}$. In particular, we have proved the first inequality in~\eqref{critical estimate front}.

\paragraph{Step 3. Upper estimate.} Let us consider
$$\widehat{u} (t,x) := \left( 1 - \frac{M}{t^{1/4}} \right) \times \frac{t^{3/2}}{\ln t} e^{-(x-2t + \frac{3}{2} \ln t - \ln \ln t)}  v\left(t,x - 2 t + \frac{3}{2} \ln t - \ln \ln t\right),$$
where $M>0$ is sufficiently large (to be specified later), and show that it is a positive super-solution for $t$ large enough and $x > 2t - \frac{3}{2} \ln t + \ln \ln t$.

Denoting $z = x - 2t + \frac{3}{2} \ln t - \ln \ln t$ for convenience, it is straightforward to  compute
\begin{eqnarray*}
 \left( 1 - \frac{M}{t^{1/4}} \right)^{-1} \times \frac{e^{z} \ln t}{t^{3/2}} \times ( \partial_t  \widehat{u} - \partial_{xx} \widehat{u} - \widehat{u} ) 
 = \left( \frac{3}{2t} - \frac{1}{t \ln t} \right) \partial_x v  + \frac{M}{4 t^{5/4}} \left( 1 - \frac{M}{t^{1/4}} \right)^{-1} v (t,z) ,
\end{eqnarray*}
for $t >0$ and $z >0$. From the following lemma, whose proof is postponed, we may choose $t_0>0$ and $M>0$ sufficiently large to make $\widehat u$ a (positive) super-solution of the KPP equation on the moving right half-line $t>t_0$, $x>2t-\frac 3 2 \ln t+\ln \ln t$.

\begin{lemma}\label{lem:critical_super}
There are $t_0>0$ and $C>0$ such that the function $v$ defined in Step 1 satisfies
$$
\frac{\partial_x v(t,x) }{v(t,x)} \geq - \frac{C}{t^{1/4}},\quad  \text{ for all } t>t_0, \ x >0.
$$
\end{lemma}

We now sketch the end of the proof of the upper estimate. First, for any $K>1$, we define
$$\overline{u} (t,x)  = \left\{
\begin{array}{ll}
1 & \ \mbox{ if }\ x < 2(t+t_0) -\frac{3}{2} \ln (t+t_0) + \ln \ln (t+t_0) + 1 , \vspace{3pt} \\
\min \{1, K \widehat{u} (t+t_0, x ) \} & \ \mbox{ if } \ x \geq 2(t+t_0) -\frac{3}{2} \ln (t+t_0) +  \ln \ln (t+t_0) + 1 .
\end{array}
\right.
$$
We have that
\begin{eqnarray*}
 && K \widehat{u} \left( t+t_0, 2 ( t +t_0) - \frac{3}{2} \ln (t+t_0) +  \ln \ln (t+t_0) + 1  \right) \\
& &\quad =  K \left( 1 - \frac{M}{(t+t_0)^{1/4}} \right) \times \frac{(t+t_0)^{3/2}}{\ln (t+t_0)} e^{- 1}  v(t+t_0, 1)
\end{eqnarray*}
which, up to increasing $t_0$ and $K$, is strictly larger than $1$
for all $t>0$ thanks to the following lemma, whose proof is also postponed. As a result, $\overline{u}$ is a generalized super-solution of the KPP equation (for all $t>0$ and $x\in\R$), for any choice of $K$ large enough.

\begin{lemma}\label{lem:postponed}
There are $C_2 > C_1 >0$ and $T>e$ such that the function $v$ defined in Step 1 satisfies, for any $t \geq T$ and $x \in (1, \ln t)$, 
$$
C_1\frac{x \ln t}{t^{3/2}} \leq v(t,x) \leq C_2\frac{ x\ln t }{t^{3/2}}.
$$
\end{lemma}

In order to apply the comparison principle, let us  check that
\begin{equation}\label{ordre-zero}
u _0 \leq \overline{u} (t_0, \cdot ) \quad \text {on } \R.
\end{equation}
Proceeding as in the noncritical case, up to increasing the constant~$K$, it suffices to check this inequality in a neighborhood of $+\infty$. With the time $t_0$ now fixed, by the strong maximum principle and the positivity of the initial data $v_0$, see \eqref{v-zero}, there exists $\alpha \in (0,1)$ small enough such that
$$
v(t ,1 )  \geq  \alpha , \quad \forall t \in [0, t_0].
$$
Moreover, recall that
$$v_0 (x) = \frac{1}{x^2}, \quad \forall x\geq 1.$$
Also, since $\underline{v}(t,x) := \frac{\alpha}{x^2}$ satisfies
$$\partial_t \underline{v} = 0 <  \partial_{xx} \underline{v},\quad x \geq 1,
$$
by the comparison principle, we get that
$$
v (t_0,x) \geq  	 \frac{\alpha}{x^2},\quad \forall
x \geq 1.
$$
In particular, for all $x\geq 2t_0-\frac 3 2 \ln t_0+\ln \ln t_0+1$,
$$
\widehat{u} (t_0, x) \geq \left( 1 - \frac{M}{t_0^{1/4}} \right) \times \frac{ t_0^{3/2}}{\ln t_0} e^{2 t_0 - \frac{3}{2} \ln t_0 + \ln \ln t_0}  \times \frac{\alpha e^{-x}}{(x-2t_0 + \frac{3}{2} \ln t_0 - \ln \ln t_0)^2}.
$$
We can thus find~$K$ large enough so that
$$K \widehat{u} (t_0, x) \geq A \frac{e^{-x}}{x^2} \geq u_0(x),$$
on a neighborhood of $+\infty$, so that \eqref{ordre-zero} does hold.

We are now in a position to apply the comparison principle, yielding
$$u (t,x) \leq \overline{u} (t+t_0, x),$$
for all $t >0$ and $x \in \mathbb{R}$. The upper estimate on the position of the level sets finally follows, as detailed below.

First, from our results in the noncritical case (say for a chosen $k$ in $(-2,-1)$ for which $r>1$) and the comparison principle, we already know that
\begin{equation}\label{eq:we already know}
\limsup_{t \to +\infty}\sup_{x \geq 2 (t+t_0) - \ln (t+t_0)} u (t,x) = 0. 
\end{equation}
Next, for any $m \in (0,1)$, select $X>1$ large enough so that $KC_2Xe^{-X}<m$ (where $C_2$ comes from Lemma \ref{lem:postponed}) and define the interval
$$
I_t := \left( 2t - \frac{3}{2}\ln t + \ln \ln t + X, 2t - \ln t \right).
$$
Then we have
\begin{eqnarray*}
&& K^{-1} \times \limsup_{t \to +\infty}\sup_{x \in I_{t+t_0}} u(t,x) \\
&& \quad \leq   \limsup_{t \to +\infty}\sup_{x \in I_{t+t_0}} \widehat{u} (t+t_0,x) \\
&& \quad \leq \limsup_{t \to +\infty}\sup_{x \in I_{t+t_0}} \frac{(t+t_0)^{3/2}}{\ln (t+t_0)} e^{-(x-2 (t +t_0) + \frac{3}{2}\ln (t+t_0) - \ln \ln (t+t_0))} \\
&&\qquad \qquad  \qquad \qquad\qquad \qquad \qquad  \times v \left(t+t_0,x -2 (t+t_0) + \frac{3}{2}\ln (t+t_0)  - \ln \ln (t+t_0) \right)\\
&& \quad \leq \limsup_{t \to +\infty}\sup_{X \leq z  \leq \ln (t+t_0)} \frac{(t+t_0)^{3/2}}{\ln (t+t_0)} e^{-z}  \times v \left(t+t_0,z\right)\\
&& \quad \leq \limsup_{t \to +\infty}\sup_{X \leq z  \leq \ln (t+t_0)} C_2 z e^{-z} \\
&& \quad \leq    C_2 X e^{-X}\\
&& \quad <K^{-1} m,
\end{eqnarray*}
where we have used Lemma \ref{lem:postponed} again for the antepenultimate inequality. Putting this together with \eqref{eq:we already know}, we infer that there is $T_m>0$ large enough so that, for all $t\geq T_m$,
$$E_m (t) \subset \left(- \infty ,  2 (t+t_0) - \frac{3}{2} \ln (t+t_0) + \ln \ln (t+t_0) + X \right),$$
which completes the proof of \eqref{location-critical-case}.

Now fixing, say, $m=\frac 12$ and proceeding similarly, we may also find that, for any $t>0$ and~$y$ such that 
$$2t- \frac{3}{2}\ln t + \ln \ln t + y \in I_{t+t_0},$$
then 
\begin{eqnarray*}
& & u\left( t,2t - \frac{3}{2} \ln t + \ln \ln t + y \right) \\
& \leq & K \frac{(t+t_0)^{3/2}}{\ln (t+t_0)} e^{-(y- 2 t_0 + \frac{3}{2}\ln \frac{t+t_0}{t} + \ln \ln t - \ln \ln (t+t_0))} \\
&&\qquad \qquad  \qquad \qquad\qquad  \times v \left(t+t_0, y- 2 t_0 + \frac{3}{2}\ln \frac{t+t_0}{t} + \ln \ln t - \ln \ln (t+t_0) \right)\\
&\leq & K  C_2 \times \left( y- 2 t_0 + \frac{3}{2}\ln \frac{t+t_0}{t} + \ln \ln t - \ln \ln (t+t_0) \right) e^{-(y- 2 t_0 + \frac{3}{2}\ln \frac{t+t_0}{t} + \ln \ln t - \ln \ln (t+t_0))}\\
&\leq  &  k_1 (y+1) e^{-y}, 
\end{eqnarray*}
for some $k_1>0$, and for all $t \geq T$ with $T>0$ large enough. Now, letting $\sigma_1 \in (0,1/2)$, we can assume up to increasing $T$ that, for all $t \geq T$,
$$y \in (2t_0 + X +1 , \sigma_1 \ln t) \Rightarrow 2t- \frac{3}{2}\ln t + \ln \ln t + y \in I_{t+t_0},$$
and then 
\begin{equation}\label{upper estimate front critical}
u \left( t, 2t - \frac{3}{2}\ln t + \ln \ln t + y \right) \leq k_1 (y+1) e^{-y},
\end{equation}
 for all $t \geq T$ and $2 t_0 + X +1  \leq y \leq \sigma_1 \ln t$. Since the left-hand term of~\eqref{upper estimate front critical} is bounded from above by $1$, and the right-hand term has a positive minimum on $[0,2t_0 + X+1]$, one can deduce that~\eqref{upper estimate front critical} holds true for all $t \geq T$ and $0 \leq y \leq \sigma_1 \ln t$, up to increasing~$k_1$. This proves the second inequality in~\eqref{critical estimate front}.
\qed
\medskip

It thus only remains to prove the three postponed lemmas on the function $v$ defined in Step 1.

\begin{proof}[Proof of Lemma~\ref{lem:critical_sub}] 
Notice that $e^{-(x-2t)} v(t,x-2t)$ solves the linearized equation
$$\partial_t u = \partial_{xx} u + u,$$
together with a Dirichlet boundary condition at $x = 2t$. Recall also that we have found in the noncritical case a family of super-solutions to this linear equation. Indeed, for any $k \in (-2,1)$ and $r = \frac{1-k}{2} >0$ (see Step~1 in subsection~\ref{ss:proof-upper} in the delay case), then 
$$\overline{u} (t,x ) = \left( 1 - \frac{M}{\sqrt{t}} \right) e^{-(x-2t + r \ln t )} \sqrt{t} \times w \left( \frac{x-2t + r \ln t}{\sqrt{t}} \right),
$$
where $w$ comes from Lemma~\ref{lem:ode-w}, is a positive super-solution to this linear equation for all $t \geq t_0$ (with $t_0$ large enough) and $x \geq 2t > 2t - r \ln t$.

By construction,
$$e^{-x} v (0,x) \leq x^k e^{-x},$$
for all $x > 1$. Therefore, proceeding similarly to the noncritical case, we may find that
$$e^{-(x-2t)} v (t,x-2t) \leq K \overline{u} (t+t_0, x),$$
for some large enough $K$, and any $t>0$, $x > 2t$. Due to the properties of $w$ from Lemma~\ref{lem:ode-w}, namely the fact that $\sup_{y \geq 0} w(y)/y < + \infty$, we get that
\begin{eqnarray*}
 e^{-(x-2t)} v (t,x-2t)   &\leq &  K e^{-(x-2 (t+ t_0) + r \ln (t+t_0))} \sqrt{t} \times w \left( \frac{x-2(t+t_0) + r \ln (t+t_0)}{\sqrt{t+t_0}} \right) \\
& \leq  & C e^{-(x-2t)}  \times t^{-r}  \times (x-2t + r \ln (t+t_0)) ,
\end{eqnarray*}
for some $C>0$. Since $r$ can be chosen arbitrarily close to $3/2$, we reach the conclusion. 
\end{proof}

\begin{proof}[Proof of Lemma~\ref{lem:critical_super}]
First, recall that
\begin{eqnarray*}
v(t,x) & = & \frac{1}{\sqrt{4 \pi t}} \int_0^{+\infty} \left(  e^{-\frac{(x-y)^2}{4t}} - e^{-\frac{(x+y)^2}{ 4t}} \right) v_0 (y) dy \\
& = & \frac{e^{-\frac{x^2}{4t}} }{\sqrt{\pi t}} \int_0^{+\infty}  e^{-\frac{y^2}{4t}} \sinh \left(\frac{xy}{2t}\right)  v_0 (y) dy .
\end{eqnarray*}
Using the second expression, we find that
$$
\partial_x v (t,x)  =  - \frac{x}{2t} v (t,x) + \frac{e^{-\frac{x^2}{4t}} }{\sqrt{\pi t}} \int_0^{+\infty}  \frac{y e^{-\frac{y^2}{4t}}}{2t} \cosh \left(\frac{xy}{2t}\right)  v_0 (y) dy  \geq  - \frac{x}{2t} v (t,x).
$$
It already follows that
$$
\frac{\partial_x v (t,x)}{v(t,x)  } \geq - \frac{1}{t^{1/4}}  \quad \text{ for all } t>0,\, 0\leq x \leq 2 t^{3/4}.
$$
Next, let us consider the case $x > 2 t^{3/4}$. On the one hand, for any $t>1$ (so that also $x - \sqrt{t}>1$), we have
\begin{eqnarray*}
v(t,x)&  \geq &  \frac{1}{\sqrt{4 \pi t}} \int_{x-\sqrt{t}}^{x+\sqrt{t}} \left(  e^{-\frac{(x-y)^2}{4t}} - e^{-\frac{(x+y)^2}{ 4t} } \right) v_0 (y) dy \\
& = & \frac{1}{\sqrt{4 \pi t}} \int_{-\sqrt{t} }^{\sqrt{t}} \left(  e^{-\frac{y^2}{4t}} - e^{-\frac{(2x+y)^2}{ 4t} } \right) \times \frac{1}{(x+y)^2} dy \\
& \geq & \frac{e^{-\frac 14}}{\sqrt{4 \pi t}} \int_{-\sqrt{t} }^{\sqrt{t}} \left(  1 - e^{-\frac{x^2+xy}{ t} } \right) \times \frac{1}{(x+y)^2} dy.
\end{eqnarray*}
But $-\sqrt t\leq y\leq \sqrt t$, $x>2t^{3/4}$ and $t>1$ imply  $\frac{x^2+xy}{ t}\geq 2t^{1/4}\geq 2$ so that  
$$
v(t,x) 
 \geq  \frac{e^{-\frac 1 4}}{x^2\sqrt{4 \pi t}}(1-e^{-2}) \int_{-\sqrt{t} }^{\sqrt{t}}  \frac{1}{(1+\frac y x)^2} dy  \geq \frac{\delta}{x^2},
$$
for some $\delta >0$. On the other hand,
\begin{eqnarray*}
\partial_x v (t,x) & = &   \frac{1}{\sqrt{4 \pi t}} \int_0^{+\infty} \left( - \frac{x-y}{2t} e^{-\frac{(x-y)^2}{4t}} + \frac{x+y}{2t} e^{-\frac{(x+y)^2}{4t} } \right) v_0 (y) dy \\
& \geq &   -  \frac{1}{\sqrt{4 \pi t}} \int_0^{+\infty} \frac{x-y}{2t} e^{-\frac{(x-y)^2}{4t}}  v_0 (y) dy \\
& = & - \frac{1}{\sqrt{4 \pi t}} (I_1 + I_2 + I_3),
\end{eqnarray*}
where $I_1$, $I_2$, and $I_3$ are obtained by integrating, respectively, over $(0,1)$, $(x(1-t^{-1/8}),x(1+t^{-1/8}))$, and the remaining intervals. 
Recalling \eqref{v-zero}, we have, for any $t\geq t_0$ with $t_0>0$ large enough and $x > 2 t^{3/4}$,
$$
I_1  =   \int_0^1 \frac{x-y}{2t} e^{-\frac{(x-y)^2}{4t}}  dy  \leq  \frac{x e^{-\frac{(x-1)^2}{4t}} }{2t},
$$
but also (in the sequel $C$ denotes a positive constant that may vary from one line to another)
\begin{eqnarray*}
I_2 & = &  \int_{x(1 - t^{-1/8}) }^{x(1 + t^{-1/8} ) } \frac{x-y}{2t} e^{-\frac{(x-y)^2}{4t}}  v_0 (y) dy \\
& = &  \int_{- x t^{-1/8}}^{x t^{-1/8}} \frac{y}{2t} e^{-\frac{y^2}{4t}} \times  \frac{1}{x^2 (1- \frac{y}{x})^2} dy \\
& \leq & \int_{- x t^{-1/8}}^{x t^{-1/8}} \frac{y}{2t} e^{-\frac{y^2}{4t}} \times \frac{ 1 + 2 \frac{y}{x}}{x^2} dy  + C \int_{- x t^{-1/8}}^{x t^{-1/8}} \frac{|y|^3}{t x^4 } e^{-\frac{y^2}{4t}}dy\\
& = &  2\int_{0}^{x t^{-1/8}} \frac{y^2}{tx^3} e^{-\frac{y^2}{4t}} dy+ C \int_{- x t^{-1/8}}^{x t^{-1/8}} \frac{|y|^3}{t x^4 } e^{-\frac{y^2}{4t}} dy \\
& \leq & \frac{2\sqrt{t}}{x^3} \int_{0}^{+\infty} y^2 e^{-\frac{y^2}{4}} dy+ \frac{ C t}{x^4} \int_{-\infty}^{+\infty} |y|^3 e^{-\frac{y^2}4 } dy \\
& \leq & \frac{C \sqrt{t}}{x^3}\quad(\text{since}\ x>2t^{3/4}),
\end{eqnarray*}
and finally
\begin{eqnarray*}
I_3 & = &  \left( \int_{1}^{x ( 1 - t^{-1/8}) }  + \int_{x ( 1 + t^{-1/8}) }^{+\infty} \right) \frac{x-y}{2t} e^{-\frac{(x-y)^2}{4t}}  v_0 (y) dy \\
& \leq & \left( \int_{1}^{x ( 1 - t^{-1/8})}  + \int_{x ( 1 + t^{-1/8}) }^{+\infty} \right) \frac{|x-y|}{2t} e^{-\frac{(x-y)^2}{4t}}   dy \\
& \leq  & \left( \int_{-\infty}^{-x t^{-1/8}}+ \int_{x  t^{-1/8} }^{+\infty} \right) \frac{|y|}{2t} e^{-\frac{y^2}{4t}}   dy \\ 
& \leq & \left( \int_{-\infty}^{-x t^{-5/8}}+ \int_{x t^{-5/8} }^{+\infty} \right) \frac{|z|}2  e^{-\frac{z^2}4}   dz \\
& \leq & C e^{-\frac{(x t^{-5/8})^2}4}.
\end{eqnarray*}
Using again the fact that $x > 2 t^{3/4}$, we have, for any $t\geq t_0$ with $t_0>0$ large enough,
$$
\frac{x^3}{\sqrt{t}} I_1 \leq \frac{x^4 e^{-\frac{(x-1)^2}{4t}}}{2t^{3/2}}\leq C t^{3/2} e^{-\sqrt{t}} \leq C,
$$
and
$$ \frac{x^3}{\sqrt{t}} I_3 \leq C \frac{x^3}{\sqrt{t}} e^{-\frac{(x t^{-5/8})^2}4} \leq C t^{7/4} e^{-t^{1/4}} \leq C.
$$
 Finally, putting all this together, we find that
$$\frac{\partial_x v(t,x)}{v(t,x)} \geq  \frac{- (I_1 + I_2 + I_3)}{\sqrt{4 \pi t} \times v (t,x)} \geq - \frac{C}{x ^3} \times \frac{x^2}{\delta} \geq - \frac{C}{t^{3/4}},$$
for any $t\geq t_0$ with $t_0>0$ large enough and $x > 2 t^{3/4}$. Putting this together with our previous estimate for $x \leq 2 t^{3/4}$, the proof of Lemma \ref{lem:critical_super} is complete.
\end{proof}

\begin{proof}[Proof of Lemma \ref{lem:postponed}]
Consider $t > e$ and $x \in (1, \ln t)$. Then, we have
\begin{eqnarray*}
v (t , x ) & = &  \frac{e^{-\frac{x^2}{4t}}}{\sqrt{4 \pi t}} \int_0^{+\infty} e^{-\frac{y^2}{4t}} \left( e^{\frac{yx}{2t}} - e^{-\frac{yx}{2t}} \right) v_0 (y) dy\\
& = &  \frac{e^{-\frac{x^2}{4t}}}{\sqrt{4 \pi}} \int_0^{+\infty} e^{-\frac{y^2}{4}} \left( e^{\frac{yx}{2\sqrt{t}}} - e^{-\frac{yx}{2\sqrt{t}}} \right) v_0 (\sqrt{t} y) dy\\
& = &  I_1 + I_2 + I_3 + I_4,
\end{eqnarray*}
where $I_i$ ($1\leq i\leq 4$) will be defined and estimated below. For the sake of expediency, in the sequel $f(t,x) \approx g(t,x)$ means that there is $C>0$ such that
$$
\frac{f(t,x)}{C g(t,x)}\to 1 \quad \text{ as $t\to +\infty$, uniformly with respect to $x \in (1, \ln t)$}.
$$
Now, recalling \eqref{v-zero}, 
\begin{eqnarray*}
I_1 & := &  \frac{e^{-\frac{x^2}{4t}}}{\sqrt{4 \pi }}  \int_0^{1/ \sqrt{t}} e^{-\frac{y^2}{4}} \left( e^{\frac{yx}{2\sqrt{t}}} - e^{-\frac{yx}{2\sqrt{t}}} \right)  dy\\
& \approx &  \int_0^{1/\sqrt{t}} \frac{yx}{ \sqrt{t}} dy \\
& \approx & \frac{x}{t^{3/2}} ,
\end{eqnarray*}
\begin{eqnarray*}
I_2 & := & \frac{e^{-\frac{x^2}{4t}}}{\sqrt{4 \pi }}   \int_{1/ \sqrt{t}}^{1} e^{-\frac{y^2}{4}} \left( e^{\frac{yx}{2\sqrt{t}}} - e^{-\frac{yx}{2\sqrt{t}}} \right)  \times \frac{1}{t y^2} dy\\
& \approx &  \int_{1/ \sqrt{t}}^{1}  \frac{x}{t^{3/2} y} dy \\
& \approx & \frac{x \ln t}{t^{3/2}},
\end{eqnarray*}
\begin{eqnarray*}
0\leq I_3 & := & \frac{e^{-\frac{x^2}{4t}}}{\sqrt{4 \pi }}   \int_1^{\sqrt{t}/x}  e^{-\frac{y^2}{4}} \left( e^{\frac{yx}{2\sqrt{t}}} - e^{-\frac{yx}{2\sqrt{t}}} \right)  \times \frac{1}{t y^2} dy\\
& \leq  &  C \int_1^{\sqrt{t}/x}  e^{-\frac{y^2}{4}}  \times \frac{x}{t^{3/2} y} dy \approx  \frac{x}{t^{3/2}} ,
\end{eqnarray*}
and, recalling also that $x \in (1,\ln t)$,
\begin{eqnarray*}
0\leq I_4 & := & \frac{e^{-\frac{x^2}{4t}}}{\sqrt{4 \pi }}   \int_{\sqrt{t}/x}^{+\infty}  e^{-\frac{y^2}{4}} \left( e^{\frac{yx}{2\sqrt{t}}} - e^{-\frac{yx}{2\sqrt{t}}} \right)  \times \frac{1}{t y^2} dy\\
& \leq  &  \frac{1}{t} \int_{\sqrt{t}/x}^{+\infty}  e^{-\frac{y^2}{8}} dy \\
&\leq & \frac 1 t \frac{x}{\sqrt t} \int_{\sqrt{t}/x}^{+\infty}  ye^{-\frac{y^2}{8}} dy\\
& \leq &  C \frac{x}{t^{3/2}} e^{-\frac{t}{8\ln ^2 t}}.
\end{eqnarray*}
As a result $v (t, x) = I_1 + I_2 + I_3 + I_4 \approx \frac{x\ln t}{t^{3/2}}$, which ends the proof.
\end{proof}

\section{Convergence of the profile}\label{s:proof of th}

In this section, we complete the proof of our main results, namely Theorems~\ref{th:bra} and~\ref{th:critical}. In both the noncritical and critical cases, we have already located the level sets, and derived some lower and upper estimates on the solution in the appropriate moving frame. Now, we address the convergence to a minimal traveling wave in that moving frame.

The proof, which relies on a Liouville-type result from~\cite{Ber-Ham-07}, closely parallels the one in~\cite[Section~4]{Ham-Nol-Roq-Ryz-13} to which we refer for further details. Notably, it proceeds exactly in the same manner for both the noncritical and critical cases. Therefore, we only provide the details in the former, and point out that in the latter, one should use Proposition~\ref{had to be done} instead of Propositions~\ref{prop:super-nancy} and~\ref{prop:lower}, and~\eqref{critical estimate front} instead of~\eqref{upper estimate front} and~\eqref{low estimate front}.

\begin{proof}[Proof of Theorem \ref{th:bra}]  First, recall from Propositions~\ref{prop:super-nancy} and~\ref{prop:lower} that, for any $0 < m < 1$, the $m$-level set  $E_m (t)$ of the solution satisfies 
\begin{equation}\label{sandwich-level-sets}
E_m (t) \subset ( 2t - r \ln t -C_m, 2t - r \ln t + C_m), \quad \forall t\geq T_m,
\end{equation}
for some $C_m >0$ and $T_m >0$.

We first set $C>0$ large enough such that
\begin{equation}\label{aaa}
Be^{-C}\le k_2\le k_1\le Be^C,
\end{equation}
where $k_1$ and $k_2$ are defined in Propositions~\ref{prop:super-nancy} and~\ref{prop:lower}, and $B$ comes from the asymptotics $U(z) \sim B z e^{-z}$ as $z \to +\infty$.
We will show that \eqref{asy} holds with this choice of $C$.

We argue by contradiction, and assume there are $\varepsilon>0$ and a
sequence of times $(t_n)_{n\in \mathbb{N}}$ such that $t_n\to+\infty$ as $n\to+\infty$ and
$$\inf_{|h|\le C}\; \Big\Vert u(t_n,\cdot)-U(\cdot-2t_n+r\ln t_n+h)\Big\Vert_{L^{\infty}(0,+\infty)}\ge \varepsilon, \quad \forall n\in \N.$$
Since $U(-\infty)=1$ and $U(+\infty)=0$, by estimate \eqref{sandwich-level-sets} on the position of the level sets, we can find $\kappa>0$ such that
\begin{equation}\label{bbb}
\inf_{|h|\le C}\; \Big(\max_{|x|\le \kappa}\Big| u(t_n,x+2t_n-r\ln t_n)-U(x+h)\Big|\Big)\ge \varepsilon, \quad \forall n\in \N.
\end{equation}
Up to extraction, the sequence of  functions
$$
u_n(t,x) :=u(t+t_n,x+2t_n-r\ln t_n), \quad t\in \R, x\in \R,
$$
converges locally uniformly in $\R^2$ to some $0\leq u_{\infty}=u_{\infty}(t,x)\leq 1$ which solves \eqref{eq} on $\R^2$. Furthermore, we have
\begin{equation}\label{ccc}
\lim_{y\to+\infty}\Big(\sup_{(t,x)\in\mathbb{R}^2,\ x\ge 2t+y}u_{\infty}(t,x)\Big)=0\quad \text{and}\quad \lim_{y\to-\infty}\Big(\inf_{(t,x)\in\mathbb{R}^2,\ x\le 2t+y}u_{\infty}(t,x)\Big)=1.
\end{equation}
Now, we fix $t\in\mathbb{R}$ and $y>0$, and define
$$
y_n:=y+r\ln \frac{t+t_n}{t_n}.
$$
For $n$ sufficiently large, we have  $t+t_n\ge 1$ and $0\leq y_n \leq \min\{ \sigma_1, \sigma_2 \} \sqrt{t+t_n}$, so that we deduce from~\eqref{upper estimate front} and \eqref{low estimate front} that
$$
k_2y_ne^{-y_n}\le u_n(t,2t+y)\le k_1(y_n+1)e^{-y_n}.
$$
Passing to the limit as $n \to +\infty$, we obtain
\begin{equation}\label{ddd}
k_2ye^{-y}\le u_{\infty}(t,2t+y)\le k_1(y+1)e^{-y}\quad \text{ for any } t\in\mathbb{R}, y\ge 0.
\end{equation}
Therefore, a Liouville-type theorem, see~\cite[Lemma~4.1]{Ham-Nol-Roq-Ryz-13} and~\cite[Theorem~3.5]{Ber-Ham-07}, implies that there exists $h_0\in \R$ such that the time-global solution~$u_{\infty}$ satisfies
\begin{equation}\label{eee}
u_{\infty}(t,x)=U(x-2t+h_0), \quad \forall (t,x)\in\mathbb{R}^2.
\end{equation}

We are now ready to complete the proof. Recall that $U(z) \sim B z e^{-z}$ as $z \to +\infty$. Thus, from~\eqref{ddd} and~\eqref{eee}, we have $k_2 \leq Be^{-h_0} \leq k_1$ which, in view of \eqref{aaa}, enforces  $|h_0|\le C$. However, $u_n(0,x)\to u_\infty(0,x)$ uniformly for $x\in[-\kappa,\kappa]$ means nothing else than
$$
\max_{|x|\le\kappa}\; \Big|u(t_n,x+2t_n-r\ln t_n)-U(x+h_0)\Big|\to 0\quad \text{ as } \ n\to+\infty,
$$
which contradicts \eqref{bbb} since $|h_0|\le C$. Therefore, the proof of Theorem~\ref{th:bra} is complete.
\end{proof}
\bigskip

\noindent{\bf Acknowledgement.}  M. Alfaro is supported by the {\it région Normandie} project BIOMA-NORMAN 21E04343 and the ANR project DEEV ANR-20-CE40-0011-01. T. Giletti is supported by the ANR projects Indyana ANR-21-CE40-0008 and Reach ANR-23-CE40-0023-01. D. Xiao is supported by the Japan Society for the Promotion of Science P-23314.

\bibliographystyle{siam}  

\bibliography{biblio}

\end{document}